\documentclass[12pt,a4paper]{article}
\usepackage[utf8]{inputenc}
\usepackage{lmodern}
\usepackage[T1]{fontenc}
\usepackage[english]{babel}
\usepackage{ifpdf}
\usepackage[usenames,dvipsnames]{xcolor}
\usepackage{graphicx}
\usepackage[shortlabels]{enumitem}
\usepackage[numbers,sort&compress]{natbib}
\usepackage{doi}
\usepackage[font=footnotesize,width=\textwidth]{caption}
\usepackage{booktabs}
\usepackage{hyperref}
\usepackage[top=1.8cm, bottom=2.8cm]{geometry}
\usepackage[affil-it]{authblk}
\usepackage[parfill]{parskip}
\usepackage[setbuildercolon]{matmatix}

\ifpdf
  \hypersetup{
    pdftitle={The Moran forest}
  }
\fi

\hypersetup{colorlinks, allcolors = {MidnightBlue}}

\newenvironment{mathlist}
{\begin{enumerate}[label={\upshape(\roman*)}, align=left, widest=iii, leftmargin=*]}
{\end{enumerate}\ignorespacesafterend}

\begingroup
 \makeatletter
 \@for\theoremstyle:=definition,remark,plain\do{%
  \expandafter\g@addto@macro\csname th@\theoremstyle\endcsname{%
  \addtolength\thm@preskip\parskip }%
 }
\endgroup

\makeatletter
\newtheorem*{rep@theorem}{\rep@title}
\newcommand{\newreptheorem}[2]{%
\newenvironment{rep#1}[1]{%
 \def\rep@title{#2 \ref*{##1}}%
 \begin{rep@theorem}}%
 {\end{rep@theorem}}}
\makeatother

\newtheorem{theorem}{Theorem}[section]
\newreptheorem{theorem}{Theorem}
\newtheorem{proposition}[theorem]{Proposition}
\newtheorem{lemma}[theorem]{Lemma}

\theoremstyle{definition}
\newtheorem{definitionX}[theorem]{Definition}

\newtheorem{remarkX}[theorem]{Remark}
\newenvironment{remark}
  {\pushQED{\qed}\remarkX}
  {\popQED\enddefinitionX}

\newcommand{\change}[1]{#1}

\newcommand{\MF}[1]{\mathscr{F}_{#1}}
\newcommand{\altoverline}[2]{\,\overline{\!#1}_{#2}}
\newcommand{\altunderline}[2]{\underline{#1\!}_{\:#2}}
\newcommand{\Tmax}{T^{\max}}
\newcommand{\Ti}{\mathscr{T}}
\newcommand{\FnUA}{\Restriction{F}{n - 1}}

\title{The Moran forest}

\begin{document}

\author[1,2]{François Bienvenu}
\author[3]{Jean-Jil Duchamps}
\author[1,2]{Félix Foutel-Rodier}

\affil[1]{\small Center for Interdisciplinary Research in Biology (CIRB),
CNRS UMR 7241,\newline Collège de France, PSL Research University, 
Paris, France}
\affil[2]{Laboratoire de Probabilités, Statistique et Modélisation (LPSM),
CNRS UMR 8001, Sorbonne Université, Paris, France}
\affil[3]{Laboratoire de mathématiques de Besançon (LmB) UMR 6623,
Université Bourgogne Franche-Comté, CNRS, F-25000 Besançon, France}

\maketitle

\begin{abstract}
Starting from any graph on $\{1, \ldots, n\}$, consider the Markov chain
where at each time-step a uniformly chosen vertex is disconnected from all of
its neighbors and reconnected to another uniformly chosen vertex. This Markov
chain has a stationary distribution whose support is the set of non-empty
forests on $\{1, \ldots, n\}$.
The random forest corresponding to this stationary distribution has interesting
connections with the uniform rooted labeled tree and the uniform attachment
tree. We fully characterize its degree distribution, the distribution of its
number of trees, and the limit distribution of the size of a tree sampled
uniformly.  We also show that the size of the largest tree is asymptotically
$\alpha \log n$, where $\alpha = (1 - \log(e - 1))^{-1} \approx 2.18$, and
that the degree of the most connected vertex is asymptotically
$\log n / \log\log n$.
\end{abstract}

\enlargethispage{-1ex}
\vspace{1ex}

\tableofcontents

\section{Introduction} \label{secIntro}


\subsection{The model} \label{secModel}

Consider a Markov chain on the space of directed graphs on $\Set{1, \ldots, n}$, for a fixed $n\geq 2$, whose
transition probabilities are defined as follows: at each time-step,
\begin{enumerate}
  \item Choose an ordered pair of distinct vertices $(u, v)$ uniformly at random.
  \item Disconnect $v$ from all of its neighbors, then add the edge $\vec{uv}$.
\end{enumerate}
Note that if $\vec{uv}$ is already the only edge attached to $v$ at time $t$, then
the graph is unchanged at time $t + 1$.  A simple example illustrating the
dynamics of this Markov chain is depicted in Figure~\ref{figExampleDynamics}.

\begin{figure}[h!]
  \centering
  \includegraphics[width=0.95\linewidth]{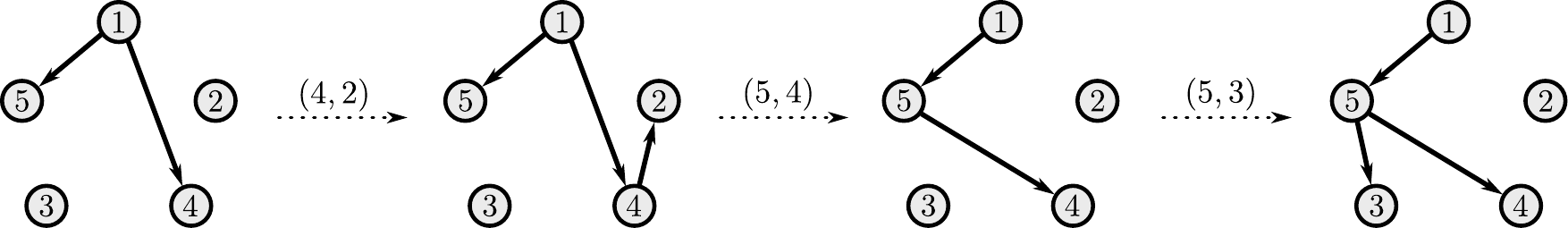}
  \caption{Example of four successive transitions of the Markov chain.
    Starting from the left-most graph, transitions are represented by dashed
    arrows decorated with the pair $(u,v)$ that is chosen uniformly at each
    step.} \label{figExampleDynamics}
\end{figure}

This Markov chain has a stationary distribution whose support is the set of
non-empty rooted forests on $\Set{1, \ldots, n}$.
\change{By \emph{rooted forest} we mean a disjoint union of rooted trees -- or,
equivalently, a directed graph where each vertex has at most one incoming edge;
any vertex $\rho$ with no incoming edge can then be seen as the root of a tree
consisting of all vertices accessible from $\rho$.}
To see why the stationary graph is a rooted forest, note that:
\begin{itemize}
  \item The graph cannot be empty because there is always an edge between
    the two vertices involved in the last transition.

  \pagebreak

  \item Starting from any graph, the chain will eventually reach a forest (for
    instance, the sequence of transitions $(1, 2), (1, 3), \ldots, (1, n)$ will
    at some point turn the graph into the star graph rooted on vertex~1).
  \item The chain cannot leave the set of forests because its
    transitions cannot make a vertex have two incoming edges.
  \item Any non-empty rooted forest is accessible from any other graph (if not clear,
    this will become apparent in Section~\ref{secConstructions}).
  \item The chain is aperiodic because it can stay in the same state.
\end{itemize}

\change{
The stationary distribution of this chain is the random forest model that
we study in this paper. We call it the \emph{Moran forest}, and use the notation
$\MF{n}$ to denote a random forest having this distribution. Our
interest in this object lies in its connection with the Moran model
of population genetics~\cite{Moran1958}. The Moran model describes the
dynamics of a population of constant size $n$ where, at each time step,
two distinct individuals are sampled uniformly at random, and the second
one is replaced by a copy of the first one. The Markov chain that we
consider thus corresponds to the family structure of extant individuals
in a Moran model. The Moran model is a central object in mathematical
population genetics~\cite{Durrett2008, Etheridge2011}, whose extensions
have been used in a variety of other contexts, including 
diversification~\cite{Hey1992, Morlon2010} and evolutionary game
theory~\cite{Nowak2006}.
}


\subsection{Main results} \label{secResults}

Our first result, which we detail in Section~\ref{secConstructions}, is that
there is a simple way to sample $\MF{n}$.
This construction enables us to study several of its
statistics, such as its number of trees (Section~\ref{secLawNumberTrees}),
its degree distribution (Section~\ref{secDegreeFixed}), and the typical size of
its trees (Section~\ref{secSizeUnifTree}). Some of these results
are presented in Table~\ref{tabStatistics}.

\begin{table}[h!]
\centering
\begin{tabular}{llll}
\toprule
  Notation & Variable &  Distribution  \\
\midrule
  $N_n$ & Number of trees & $\displaystyle \sum_{\ell = 1}^n I_\ell$,
    where $I_\ell \sim \mathrm{Ber}\mleft(\tfrac{\ell - 1}{n - 1}\mright)$ \\ \addlinespace 
  $D$ & Asymptotic degree & $\displaystyle \mathrm{Ber}(1 - U) + \mathrm{Poisson}(U)$, \\
      & distribution      & where $U \sim \mathrm{Unif}\mleft(\ClosedInterval{0, 1}\mright)$ \\ \addlinespace
  $T^U$ & Asymptotic size of & $\mathrm{Geometric}(e^{-X})$,  \\
        & a uniform tree     & where $X \sim 2xdx$ on $\ClosedInterval{0, 1}$ \\
\bottomrule
\end{tabular}
\caption{Some statistics of the Moran forest, for fixed $n$ in the case of
the number of trees, and as ${n \to \infty}$ for the degree and the size of a
uniform tree. Note that the degree also has a simple, explicit distribution
for fixed $n$ (see Proposition~\ref{propLawDegree}); \change{also, the
Bernoulli and the Poisson r.v.\ appearing in the sum correspond respectively
to the in- and out-degrees.} The Bernoulli variables $I_\ell$ used to
describe the distribution of $N_n$ are independent and,
conditional on $U$, so are the Bernoulli and Poisson variables used for the
distribution of~$D$.}
\label{tabStatistics}
\end{table}

In Section~\ref{secUnifLabeledTrees}, we show that the Moran forest is closely
linked to the uniform rooted labeled tree. Specifically, we prove the following
theorem.

\begin{theorem} \label{thmConstructionUniformTrees}
Let $\mathcal{T}$ be a uniform rooted tree on $\Set{1, \ldots, n - 1}$.
From this tree, build a forest $\mathcal{F}$ on $\Set{1, \ldots, n}$ according
to the following procedure:
\begin{enumerate}
  \item Remove all decreasing edges from $\mathcal{T}$ (that is, edges
    $\vec{uv}$ pointing away from the root such that $u > v$).
  \item Add a vertex labeled $n$ and connect it to a uniformly
    chosen vertex of~$\mathcal{T}$
  \item Relabel vertices according to a uniform permutation of
    $\Set{1, \ldots, n}$.
\end{enumerate}
Then, the resulting forest $\mathcal{F}$ has the law of the Moran forest $\MF{n}$.
\end{theorem}

Finally, we study the asymptotic concentration of the largest degree and of
the size of the largest tree of $\MF{n}$. The following theorems are proved in
Sections~\ref{secMaxDegree} and~\ref{secMaxTree}, respectively.


\begin{theorem} \label{thmMaxDegree}
Let $D_n^{\max}$ denote the largest degree of~$\MF{n}$.
Then,
\[
  D_n^{\mathrm{max}} =\; \frac{\log n}{\log\log n} \;+\;
  \big(1 + o_\mathrm{p}(1)\big)\, \frac{\log n\,\log\log\log n}{(\log\log n)^{2}},
\]
where $o_\mathrm{p}(1)$ denotes a sequence of random variables that goes to $0$ in
probability.
\end{theorem}

\begin{theorem} \label{thm:largestTree}
Let $\Tmax_n$ denote the size of the largest tree of $\MF{n}$. Then,
\[
  \Tmax_n = \alpha\big(\log n - (1+o_{\mathrm{p}}(1))\log\log n\big),
\]
where $\alpha = (1-\log(e-1))^{-1} \approx 2.18019$.
\end{theorem}

\section{Sampling of the stationary distribution} \label{secConstructions}

\subsection{Backward construction} \label{secBackward}

Consider an i.i.d.\ sequence $((V_t, W_t),\, t \in \Z)$, where
$(V_t, W_t)$ is uniformly distributed on the set of ordered pairs of distinct
elements of $\Set{1, \ldots, n}$. These variables are meant to encode the
transitions of the chain: $W_t$ represents the vertex that is disconnected at
step $t$, and $V_t$ the vertex to which $W_t$ is then connected. 
We now explain how to construct a chain $(\MF{n}(t),\, t \in \Z)$ of forests
by looking at the sequence $((V_t, W_t),\, t \in \Z)$ backwards in time.

Fix a focal time $t \in \Z$. For each vertex $w$, let us denote by
\[
    \tau_t(w) \defas \max \Set{s \le t \suchthat W_s = w}
\]
the last time before $t$ that $w$ was chosen to be disconnected, and define
\[
  m_t(w) \defas V_{\tau_t(w)}
\]
to be the vertex to which it was then reconnected.  We refer to the time
$\tau_t(w)$ as the \emph{birth time} of $w$, and to the vertex~$m_t(w)$ as its
\emph{mother}. Note that the variables $(\tau_t(w),\, 1 \le w \le n)$ are
independent of $(m_t(w),\, 1 \le v \le n)$.

Now, for each $s \le t$, let the vertices be in one of two states, \emph{active}
or \emph{inactive}, as follows: vertex $w$ is active at times~$s$ such
that ${\tau_t(w) \le s \le t}$, and inactive at times $s < \tau_t(w)$.
Finally, let $\MF{n}(t)$ be the forest obtained by connecting each vertex
$w$ to its mother if the mother is active at the time of birth of $w$, that is,
\[
  \text{there is an edge from $m_t(w)$ to $w$} \;\iff\; \tau_t(m_t(w)) < \tau_t(w).
\]
This procedure is illustrated in Figure~\ref{figBackwardConstruction}.

\begin{figure}[ht]
  \centering
  \includegraphics[width=0.85\linewidth]{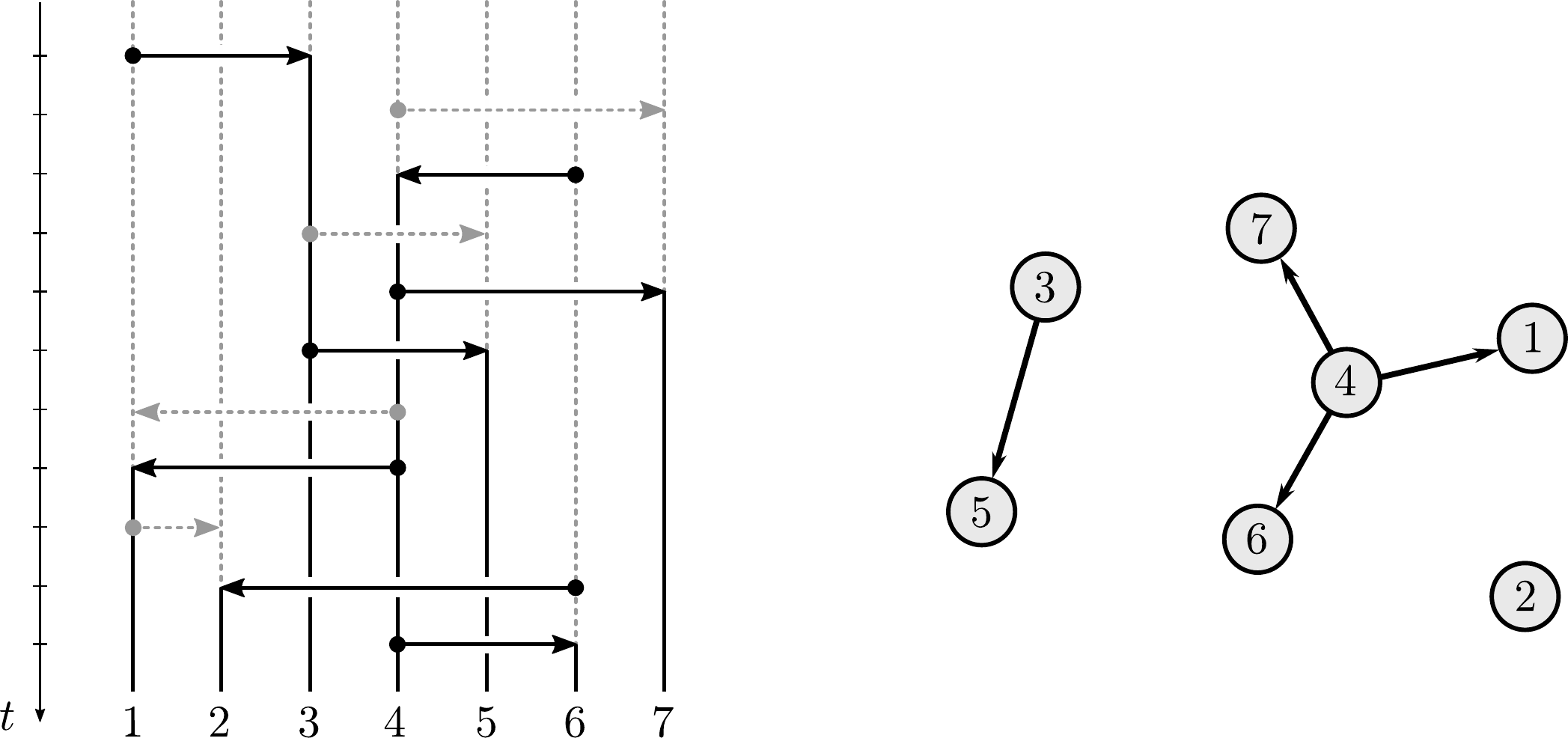}
  \caption{Illustration of the backward construction. Each vertex corresponds
  to a vertical line. A pair $(V_t, W_t)$ is represented by an arrow
  $V_t \to W_t$. The line representing a vertex is solid black when that
  vertex is active, and dashed grey when it is inactive. Arrows pointing to
  inactive vertices are represented in dashed grey because they have no impact
  on the state of the graph at the focal time: their effect has been erased
  by subsequent arrows.}
  \label{figBackwardConstruction}
\end{figure}

Let us show that the chain $(\MF{n}(t),\, t \in \Z)$ has the same
transitions as the chain described in the introduction. 
First, note that for $w \ne W_t$ we have $\tau_t(w) = \tau_{t-1}(w)$,
and thus $m_t(w) = m_{t-1}(w)$. As a result, edges that do not
involve $W_t$ are the same in $\MF{n}(t)$ and in $\MF{n}(t - 1)$.
Now, $\tau_t(W_t) = t$, so that $W_t$ is always inactive as a mother in the
construction of $\MF{n}(t)$, and $m_t(W_t) = V_t$ with $\tau_t(V_t) < t$,
so that $W_t$ is linked to $V_t$ in $\MF{n}(t)$.  In other words, $\MF{n}(t)$
is obtained from $\MF{n}(t - 1)$ by disconnecting $W_t$ from its neighbors, and
then connecting it to $V_t$. This corresponds to the transitions of the
chain described in the introduction.

Finally, $(\MF{n}(t),\, t \in \Z)$ is stationary by construction, and thus
$\MF{n}(t)$ is distributed as the Moran forest for all time $t \in \Z$.

\subsection{Uniform attachment construction} \label{secUAC}

We now give a forward-in-time variant of the construction described in
the previous section. This forward-in-time procedure, which we call the
\emph{uniform attachment construction} (UA construction for short), is our main
tool to study $\MF{n}$ and will be used throughout the rest of the
paper.

In the following, we fix $n\geq 2$, since the forest $\MF{n}$ is not defined for $n=1$.
Let $(U_n(\ell),\, 1 \le \ell \le n)$ be a vector of independent variables
such that $U_n(\ell)$ is uniformly distributed on $\Set{1, \dots, n}
\setminus \Set{\ell}$. Consider the forest $\MF{n}^*$ on $\Set{1, \dots, n}$
obtained by setting 
\[
  \text{there is an edge from $k$ to $\ell$, with $k < \ell$} \iff U_n(\ell) = k.
\]
We will show that, after relabeling the vertices of $\MF{n}^*$
according to a uniform permutation of $\Set{1, \ldots, n}$, we obtain the
Moran forest. Before this let us make a few remarks.

First, it will be helpful to think of the construction of $\MF{n}^*$
as a sequential process where, starting from a single vertex labeled~1, 
for $\ell = 2, \ldots, n$ we add a new vertex labeled $\ell$
and connect it to $U_n(\ell)$ if $U_n(\ell) < \ell$. See Figure~\ref{figUAC}.
This will make the link with some well-known stochastic processes more intuitive.
This also explains that we speak of the \emph{$\ell$-th vertex
in the UA construction} to refer to vertex $\ell$ in~$\MF{n}^*$.

Second, the edges of $\MF{n}^*$ are by construction increasing,
in the sense that every edge $\vec{uv}$ in the graph is such that $u < v$.

Rooted trees that have only increasing edges are known as
\emph{recursive trees}~\cite{Drmota2009}, and forests of recursive trees have
been called \emph{recursive forests}~\cite{Balinska1994}.
Recursive trees have been studied extensively. In
particular, the uniform attachment tree, which corresponds to the uniform
distribution over the set of recursive trees, has received much
attention~\cite{Bergeron1992, Meir1974, Mahmoud1991}.
However, the random forest $\MF{n}^*$ does not seem to correspond
to any previously studied model of random recursive forest (in particular, it
is not uniformly distributed over the set of recursive forests).

\begin{figure}[h!]
  \centering
  \includegraphics[width=0.9\linewidth]{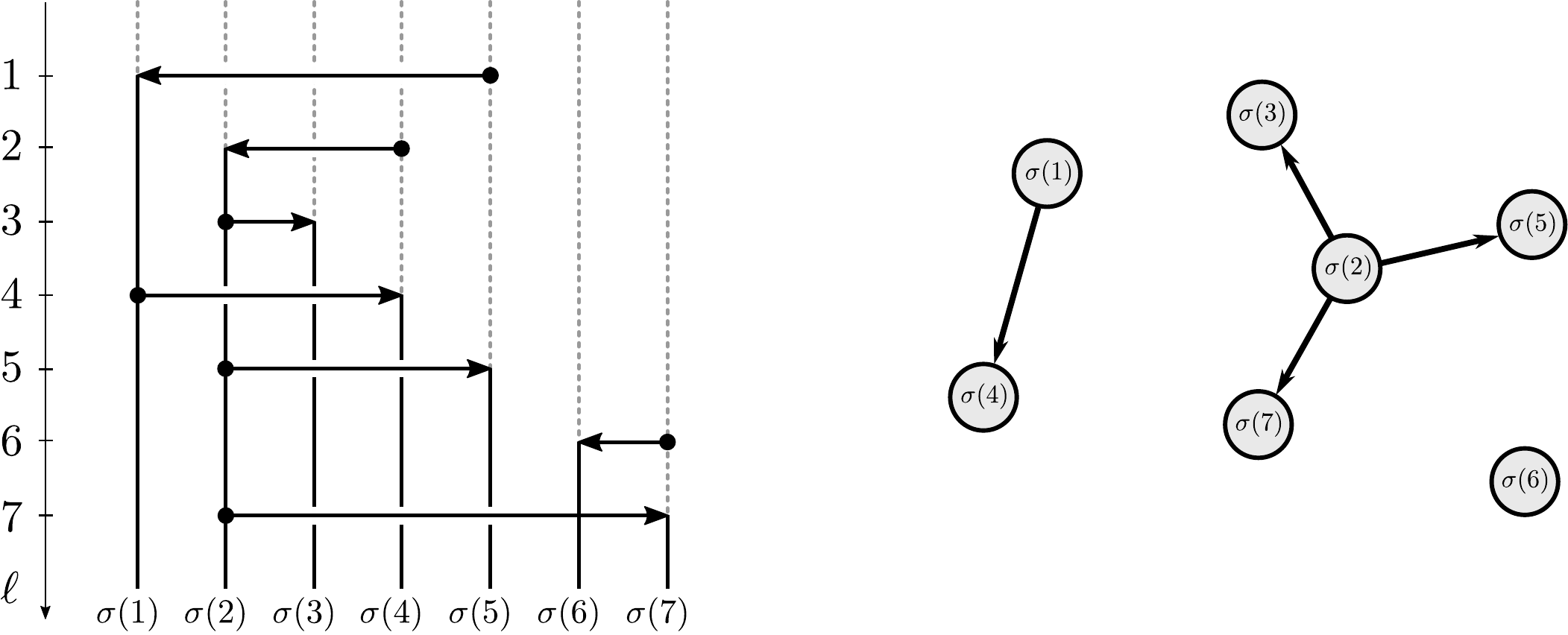}
  \caption{Illustration of the uniform attachment construction for $n = 7$ and
  the vector $(U_n(1), \ldots, U_n(n)) =(5, 4, 2, 1, 2, 7, 2)$.
  The $\ell$-th vertical line from the left corresponds to
  vertex~$\sigma(\ell)$ (i.e, in the sequential vision, to the $\ell$-th vertex
  that is added). $U_n(\ell)$ is represented by the arrow pointing from the
  $U_n(\ell)$-th line to the $\ell$-th one at time $\ell$.
  Compare this with Figure~\ref{figBackwardConstruction}: the vertical lines
  corresponding to the vertices have been reordered in increasing order of
  their birth time, and the grey arrows that left no trace on the graph at the
  focal time have been removed.} \label{figUAC}
\vspace{-1.5ex} 
\end{figure}

\begin{proposition}
The random forest obtained by relabeling the vertices of $\MF{n}^*$ according
to a uniform permutation of $\Set{1, \ldots, n}$ is distributed as the Moran
forest.
\end{proposition}

\begin{proof}
Consider the forest $\MF{n}(0)$ built from the variables
$((V_t, W_t),\, t \in \Z)$ in the previous section. To ease notation, we will
omit the subscript in $\tau_0$ and $m_0$.
The proof hinges on the fact that there is a natural coupling of $\MF{n}(0)$
with a forest $\MF{n}^*$ having the aforementioned distribution, in a way such
that conditional on $\MF{n}^*=\mathcal{F}$, $\MF{n}(0)$ is a uniform relabeling
of $\mathcal{F}$.

Let us relabel the vertices in increasing order of their
birth time: since the variables $(\tau(v),\, 1 \le v \le n)$ are all distinct,
there exists a unique permutation $\sigma$ of $\Set{1, \dots, n}$ such that
\[
  \tau(\sigma(1)) < \dots < \tau(\sigma(n)).
\]
In words, $\sigma(\ell)$ is the $\ell$-th vertex that was born in the
construction of $\MF{n}(0)$. Using the new labeling, let us denote its
birth time by $\tau^*(\ell) = \tau(\sigma(\ell))$ and its mother by
$m^*(\ell) = \sigma^{-1}(m(\sigma(\ell)))$.
  
Now, for every vertex $v = \sigma(\ell)$,
\begin{align*}
  \text{there is an edge from $m(v)$ to $v$ in $\MF{n}(0)$}\;
  &\iff\; \tau(m(v)) < \tau(v)\\
  &\iff\; \tau^*(m^*(\ell)) < \tau^*(\ell)\\
  &\iff\; m^*(\ell) < \ell \,.
\end{align*}
Thus, if we set $U_n(\ell) = m^*(\ell)$ in the construction of $\MF{n}^*$
then $\ell$ is connected to $m^*(\ell)$ if and only if
$v = \sigma(\ell)$ is connected to $m(v)=\sigma(m^*(\ell))$ in $\MF{n}(0)$. Therefore, to finish
the proof we have to show that:
\begin{mathlist}
\item The variables $(m^*(\ell),\, 1 \leq \ell \leq n)$ are independent and such that
    $m^*(\ell)$ is uniformly distributed on $\Set{1, \dots, n} \setminus \Set{\ell}$.
  \item The permutation $\sigma$ is uniform and independent of
    $(m^*(\ell),\, 1 \leq \ell \leq n)$.
\end{mathlist}

First, note that by construction the variables $(m(v),\,1\leq v\leq n)$
are independent and that for each $v$, $m(v)$ is uniformly distributed on
$\Set{1, \dots, n} \setminus \Set{v}$.
Since the permutation $\sigma$ depends only on the variables
$(\tau(v),\, 1 \le v \le n)$, which are independent of
$(m(v),\, 1 \le v \le n)$, we see that $\sigma$ is independent of
${(m(v),\, 1 \le v \le n)}$. Moreover, the variables $(\tau(v),\, 1 \le v \le n)$
are exchangeable so the permutation $\sigma$ is uniform. Now, for any fixed
permutation $\pi$ of $\Set{1, \ldots n}$ and any fixed map
$f : \Set{1, \ldots, n} \to \Set{1, \ldots, n}$ such that $f(\ell) \neq \ell$
for all $\ell$, we have
\begin{align*}
  \Prob{\sigma = \pi, \, m^* = f} \;
  &=\; \Prob{\sigma = \pi,\, m = \pi \circ f \circ \pi^{-1}}\\
  &=\; \frac{1}{n!} \frac{1}{(n-1)^n}\,, 
\end{align*}
concluding the proof.
\end{proof}

One might hope to obtain a complete description of the distribution of
the Moran forest from the UA construction. Indeed, let us denote by
$(S_1, \dots, S_{N_n})$ the sizes of the trees in the Moran forest,
labeled in decreasing order of their sizes, and where $N_n$ is the number
of trees. Then it is clear from the UA construction, that conditional on
the vector $(S_1, \dots, S_{N_n})$, the trees of the Moran
forest are independent uniform attachment trees. Therefore, the
study of the Moran forest reduces to that of the distribution of 
$(S_1, \dots, S_{N_n})$. In the terminology of exchangeable partitions, 
we need to study the \emph{exchangeable partition probability function}
(EPPF) of the Moran forest~\cite{Pitman2006}. However, we could not find
any closed expression for this EPPF. Note that the Moran forest is
\emph{not} sampling consistent, i.e., the restriction of $\MF{n+1}$ to 
$\Set{1, \dots, n}$ is not distributed as $\MF{n}$. Therefore the
EPPF of the Moran forest cannot be obtained through a Chinese restaurant
process. 

\pagebreak 

\section{Number of trees} \label{secNumberTrees}

\subsection{Law of the number of trees} \label{secLawNumberTrees}

In the UA construction, let $I_\ell = \Indic{U_n(\ell) < \ell}$ be the
indicator variable of the event ``the $\ell$-th vertex was linked to a
previously added vertex''. The variables $(I_1, \ldots, I_n)$ are thus
independent Bernoulli variables such that
\[
  I_\ell \;\sim\; \mathrm{Bernoulli}\mleft(\tfrac{\ell - 1}{n - 1}\mright)\,.
\]
With this notation, the number of edges $\Abs{E_n}$ and the number
of trees $N_n$ are
\[
  \Abs{E_n} = \sum_{\ell = 1}^{n} I_\ell
  \quad\text{and}\quad
  N_n = \sum_{\ell = 1}^{n} (1 - I_\ell) \,.
\]
Moreover, since $I_\ell \overset{d}{=} 1 - I_{n - \ell+1}$, we see that
\[
  \Prob{N_n = k} \;=\; \Prob{N_n = n - k} \;=\; \Prob{\Abs{E_n} = k}\,,
\]
that is, the number of trees and the number of edges have the
same, symmetric distribution. In consequence, from now on we only use the
notation $N_n$ and refer to it as the number of trees of $\MF{n}$ when stating
our results -- even though we sometimes work with the number of edges in the
proofs.

From the representation of $N_n$ as a sum of independent Bernoulli variables,
we immediately get the following result.

\begin{proposition} \label{propNumberTrees}
Let $N_n$ denote the number of trees of $\MF{n}$.
\begin{mathlist}
  \item $\Expec{N_n} = \dfrac{n}{2}$.
  \item $\Var{N_n} = \dfrac{n(n - 2)}{6 (n - 1)}$.
  \item $\displaystyle G_{N_n}(z) \defas \Expec{z^{N_n}} =
    \prod_{k = 1}^{n - 1} \mleft(1 + \tfrac{k}{n - 1}(z - 1)\mright)$.
\end{mathlist}
\end{proposition}

The representation of $N_n$ as a sum of independent Bernoulli variables
also makes it straightforward to get the following central limit theorem.

\begin{proposition} \label{propCLTNumberTrees}
Let $N_n$ denote the number of trees of $\MF{n}$. Then,
\[
  \frac{N_n - n/2}{\sqrt{n/6}} \;\tendsto[d]{n\to \infty}\;
  \mathcal{N}(0, 1) \,.
\]
\end{proposition}

\begin{proof}
This is an immediate consequence of the Lyapunov CLT for triangular arrays of
independent random variables. Indeed, $\Expec{\Abs{I_\ell - \Expec{I_\ell}}^3} \leq 1$.
Therefore,
\[
   \frac{1}{n^{3/2}} \sum_{\ell = 1}^n 
  \Expec{\Abs{I_\ell - \Expec{I_\ell}}^3}  \leq \frac{1}{\sqrt{n}}
  \;\tendsto{n \to \infty}\; 0\,,
\]
and the result follows, e.g., from {Corollary~11.1.4} in \cite{Athreya2006}.
\end{proof}

\subsection{Link with uniform labeled trees} \label{secUnifLabeledTrees}

As announced in the introduction, there is a strong connection between the Moran
forest and uniform labeled trees. Our starting point is the
following observation about the probability generating function of $N_n$. First,
\change{we can rewrite point~(iii) of Proposition \ref{propNumberTrees} as
\begin{align*}
  G_{N_n}(z)
  \;&=\; \frac{z}{(n - 1)^{n - 2}}\prod_{k = 1}^{n - 2} \big(n - 1 - k + k z\big)\\[1ex]
  \;&=\; \sum_{k = 0}^{n - 2} \frac{a(n - 1, k)}{(n - 1)^{n - 2}}\,z^{k + 1}, 
\end{align*}
}
where
\[
  \sum_{k = 0}^{n - 2} a(n - 1, k)\, z^k \;=\;
  \prod_{k = 1}^{n - 2} \big(n - 1 - k + k z\big). 
\]
Second, the coefficients of this polynomial have a simple combinatorial
interpretation: $a(n - 1, k)$ is the number of rooted
trees on $\Set{1, \ldots, n - 1}$ with $k$ increasing edges, where an edge
$\vec{uv}$ pointing away from the root is said to be increasing if ${u < v}$.
This fact is known in the literature as a consequence of the more general
Theorem~1.1 of~\cite{Egecioglu1986} (see also {Example~1.7.2}
in~\cite{Drake2008} and {Theorem~9.1} in~\cite{Gessel2006}).

This simple observation already gives us the following proposition.

\begin{proposition} \label{propPMFNumberTrees}
The probability mass function of the number of trees of $\MF{n}$ is
\[
  \Prob{N_n = k} \;=\; \frac{a(n - 1, k - 1)}{(n - 1)^{n - 2}} \,,
\]
where $a(n, k)$ is the number of rooted trees on $\Set{1, \ldots n}$ with $k$
increasing edges (sequence \href{https://oeis.org/A067948}{A067948} of the
On-Line Encyclopedia of Integer Sequences~\cite{OEIS}).
\end{proposition}

\change{
Looking for a bijective proof of Proposition~\ref{propPMFNumberTrees} naturally
leads to the more general Theorem~\ref{thmConstructionUniformTrees}, which
states that the Moran forest $\MF{n}$ can be obtained from a uniform rooted
tree on $\Set{1, \ldots, n - 1}$, denoted by $\mathcal{T}$, using the following
procedure:
\begin{enumerate}
  \item Remove all decreasing edges from $\mathcal{T}$ (that is, edges
    $\vec{uv}$ pointing away from the root such that $u > v$).
  \item Add a vertex labeled $n$ and connect it to a uniformly
    chosen vertex of~$\mathcal{T}$.
  \item Relabel vertices according to a uniform permutation of
    $\Set{1, \ldots, n}$.
\end{enumerate}
}

\begin{proof}[Proof of Theorem~\ref{thmConstructionUniformTrees}]
In the UA construction, let $\FnUA$ denote the forest
obtained after the addition of $n - 1$ vertices, before their relabeling.
After this, the $n$-th vertex will be linked to a uniformly chosen vertex
of $\FnUA$. As a result, to prove the theorem it suffices to show that
$\FnUA$ has the same law as the forest obtained from
$\mathcal{T}$ by removing its decreasing edges.

To do so, we couple $\FnUA$ and $\mathcal{T}$ in such a way
that the edges of $\FnUA$ are exactly the increasing edges
of $\mathcal{T}$. Formally, $\FnUA$ is a deterministic function of the random
vector $\mathbf{U} = (U_n(2), \ldots, U_n(n - 1))$. Moreover, $\mathbf{U}$ is
uniform on the set
\[
  \mathscr{S}^\star_{n - 1} \;=\;
  \Set{\mathbf{u} \in \Set{1, \ldots, n}^{\Set{2, \ldots, n - 1}} \suchthat \, 
       u_\ell \neq \ell}\,.
\]
Thus, to end the proof it is sufficient to find a bijection $\Phi$ from
$\mathscr{S}^\star_{n - 1}$ to the set of rooted trees on $\Set{1, \ldots, n-1}$
and such that
\[
  k\ell \in \FnUA(\mathbf{u}) \;\iff\;
  k\ell \text{ is an increasing edge of } \Phi(\mathbf{u}) \,.
\]

First, let
\[
  \mathscr{S}_{n - 1} \;=\; \Set{1, \ldots, n - 1}^{\Set{2, \ldots, n - 1}}
\]
and consider the bijection
$\Theta : \mathscr{S}^{\star}_{n - 1} \to \mathscr{S}_{n - 1}$ defined by
\[
  \Theta\mathbf{u} : \ell \mapsto u_\ell - \Indic{u_\ell > \ell} \,.
\]
Importantly, note that $\Theta$ does not modify the entries of $\mathbf{u}$ that
correspond to edges of $\FnUA(\mathbf{u})$, that is, for all $k < \ell$,
\[
  k\ell \in \FnUA(\mathbf{u}) \;\iff\;
  u_\ell = k \;\iff\;
  (\Theta\mathbf{u})(\ell) = k \,.
\]
As a result, it remains to find a bijection $\Psi$ from
$\mathscr{S}_{n - 1}$ to the set of rooted trees on $\Set{1, \ldots, n-1}$
such that 
\[
  u_\ell < \ell \;\iff\; u_\ell \text{ and } \ell
  \text{ are linked by an increasing edge in }\Psi(\mathbf{u}) \,.
\]
This bijection will essentially be that used in~\cite{Egecioglu1986}, which can
itself be seen as a variant of Joyal's bijection~\cite{Joyal1981, AignerZiegler2018}. 

Let $\mathscr{G}_\mathbf{u}$ be the directed graph on
$\Set{1, \ldots, n - 1}$ obtained by putting a directed edge going from
$u_\ell$ to $\ell$ for all $\ell \geq 2$.

If $\mathscr{G}_\mathbf{u}$ has no cycle or self-loop, then it is a tree.
Moreover, the orientation of its edges uniquely identify vertex~1 as its root.
Thus we set $\Psi(\mathbf{u}) = \mathscr{G}_\mathbf{u}$.

If $\mathscr{G}_\mathbf{u}$ is not a tree, set $\mathscr{C}_0 = \Set{1}$ and
let $\mathscr{C}_1, \ldots, \mathscr{C}_k$ denote the cycles of
$\mathscr{G}_\mathbf{u}$, taken in increasing order of their largest
element and treating self-loops as cycles of length~1.
Note that because each vertex has exactly one incoming edge, except
for vertex~1 which has none, these cycles are vertex-disjoint and directed.

\begin{figure}[h!]
  \centering
  \includegraphics[width=0.55\linewidth]{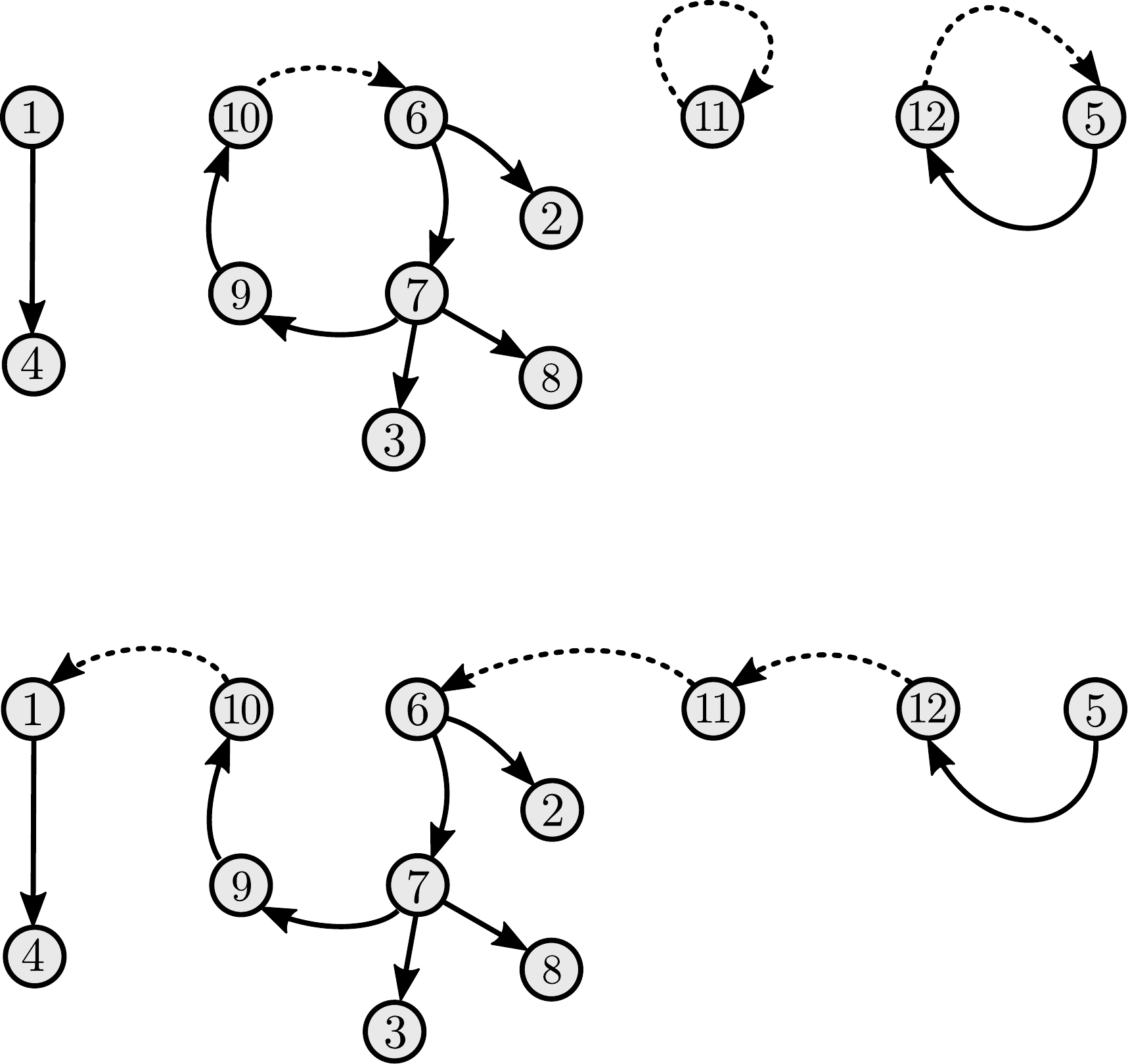}
  \caption{Example of construction of $\Phi(\mathbf{u})$, for 
  $\mathbf{u} = (7, 8, 1, 13, 11, 6, 7, 7, 9, 12, 5)$. Applying $\Theta$ yields
  $\mathbf{u}' = \Theta\mathbf{u} = (6, 7, 1, 12, 10, 6, 7, 7, 9, 11, 5)$.
  The directed graph $\mathscr{G}_{\mathbf{u}'}$ encoding $\mathbf{u}'$
  is represented on top. Its cycles are $\mathscr{C}_1 = (10, 6, 7, 9)$,
  $\mathscr{C}_2 = (11)$ and $\mathscr{C}_3 = (12, 5)$, and we set
  $\mathscr{C}_0 = (1)$. The edges $\vec{m_is_i}$ are dashed. Rewiring
  them as described in the main text turns
  $\mathscr{G}_{\mathbf{u}'}$ into the rooted tree $\Psi(\mathbf{u}')$
  represented on bottom. No information is lost when turning the cycles 
  $(1)(10, 6, 7, 9)(11)(12, 5)$ into the path going from $5$ to $1$
  encoded by the word $(1,10,6,7,9,11,12,5)$, because
  the left-to-right maxima of that word -- here $1$, $10$,
  $11$ and $12$ -- each mark the start of a new cycle.
  } \label{figExample}
\end{figure}

To turn $\mathscr{G}_\mathbf{u}$ into a tree, set $s_0 = 1$ and for $i \geq 1$
let $m_i$ denote the largest element of $\mathscr{C}_i$ and
$\vec{m_is_i}$ its out-going edge in $\mathscr{C}_i$. With this notation,
for $i = 1, \ldots, k$ remove the edge $\vec{m_is_i}$ from
$\mathscr{G}_\mathbf{u}$ and replace it by $\vec{m_is_{i - 1}\!\!\!\!\!}\,\,\,\,\,$. Note that
\begin{itemize}
  \item This turns $\mathscr{C}_0 \sqcup \cdots \sqcup \mathscr{C}_k$
    into a directed path $\mathscr{P}$ going from $s_k$ to 1.
  \item Because $m_i = \max \mathscr{C}_i$ and that $1 < m_1 < \cdots < m_k$,
    every edge $\vec{m_is_i}$ was non-increasing and has been
    replaced by the decreasing edge $\vec{m_is_{i - 1}\!\!\!\!\!}\,\,\,\,\,$.
\end{itemize}
Therefore, this procedure turns $\mathscr{G}_\mathbf{u}$ into a tree
$\Psi(\mathbf{u})$ rooted in $s_k$, without modifying its increasing edges.
Consequently, the increasing edges of $\Psi(\mathbf{u})$ are exactly the pairs
$k\ell$ for which $k = u_\ell < \ell$.

To see that $\Psi$ is a bijection, it suffices to note that
the cycles $\mathscr{C}_0, \ldots, \mathscr{C}_k$ can be recovered unambiguously
from the path $\mathscr{P}$ going from the root to vertex~1.
Indeed, writing this path as the word ${1 m_1 \cdots s_1 m_2 \cdots s_k}$, the
$m_i$ are exactly the left-to-right maxima of that word.

Setting $\Phi = \Psi \circ \Theta$ thus gives us the bijection that we were
looking for, concluding the proof.
\end{proof}

\section{Degrees} \label{secDegrees}

\subsection{Degree of a fixed vertex} \label{secDegreeFixed}

Using the UA construction and the notation from Section~\ref{secUAC}, let us
denote~by
\begin{itemize}
  \item $I_\ell = \Indic{U_n(\ell) < \ell}$ the indicator variable of the event
    ``the $\ell$-th vertex has an incoming edge linking it to a previously added vertex''.
  \item $X_\ell^{(v)} = \Indic{U_n(\ell) = \sigma^{-1}(v)}$ the indicator
    variable of the event ``the $\ell$-th vertex is linked to vertex $v$''.
  \item $B_v = \sigma^{-1}(v)$ the step of the construction at which vertex $v$
    is added.
\end{itemize}
With this notation, the degree of vertex~$v$ is
\[
  D_n^{(v)} \,=\;
  I_{B_v} \;+ \sum_{\ell = B_v + 1}^n \!\!\! X_{\ell}^{(v)},
\]
\change{where $I_{B_v}$ is the in-degree, and $\sum_{\ell = B_v + 1}^n\! X_{\ell}^{(v)}$ is the out-degree of vertex $v$.}
Moreover, conditional on $\{B_v=b\}$, $(X_{b+1}^{(v)}, \ldots, X_n^{(v)})$ are i.i.d.\ Bernoulli variables with
parameter ${1/(n - 1)}$ and $I_b$ is a Bernoulli variable with parameter
$\frac{b - 1}{n-1}$ that is independent of
$\smash{(X_{b + 1}^{(v)}, \ldots, X_{n}^{(v)})}$. As a result, conditional on $B_v$
and writing $L_v$ for $n - B_v$,
\[
  D_n^{(v)} \,\overset{d}{=}\;
  \mathrm{Ber}\mleft(1 - \tfrac{L_v}{n - 1}\mright) \;+\;
  \mathrm{Bin}\mleft(L_v,\, \tfrac{1}{n - 1}\mright) \,, 
\]
where the Bernoulli and the binomial variables are independent conditional
on~$L_v$ (here, as in similar expressions in the rest of this
document, the $\mathrm{Ber}$ and $\mathrm{Bin}$ notation refers to random
variables with the corresponding distribution, not the distributions
themselves).  Using that $L_v$ is uniformly distributed on $\Set{0, \ldots, n -
1}$, the mean, variance and probability generating function of $D_n^{(v)}$ are
obtained by routine calculations.

\begin{proposition} \label{propLawDegree}
Let $D_n$ be the degree of a fixed vertex of $\MF{n}$. Then,
\begin{mathlist}
  \item $\Expec{D_n} = 1$.
  \item $\Var{D_n} = \dfrac{2(n - 2)}{3(n - 1)}$.
  \item $\displaystyle G_{D_n}(z) \defas \Expec*{}{z^{D_n}} =
    \frac{1}{n} \sum_{\ell = 0}^{n - 1}
    \mleft(1 + (1 - \tfrac{\ell}{n - 1})(z - 1)\mright)
    \mleft(1 + \tfrac{1}{n - 1}(z - 1)\mright)^\ell$.
  \item[\textup{(iii')}] \label{item:GDn-2} $\displaystyle G_{D_n}(z) = 
    2\mleft(1 - \frac{1}{n}\mright)
    \frac{\mleft(1 + \frac{z - 1}{n - 1}\mright)^n - 1}{z\;-\;1} - 1$.
\end{mathlist}
\end{proposition}

\begin{remark}
\change{Note that, conditional on $L_v$, the probability that $v$ has a (unique) incoming edge is $1-\frac{L_v}{n-1}$, and its mean out-degree is $\frac{L_v}{n-1}$.}
Therefore, summing the two we have $\Expec*{}{D_n^{(v)} \given L_v} = 1$, that is,
the average degree of a vertex is independent of the step at which it was
added in the UA construction.
\end{remark}

\begin{proposition} \label{propLimitLawDegree}
The degree $D_n$ of a fixed vertex of $\MF{n}$ converges in distribution to the
variable $D$ satisfying:
\begin{mathlist}
  \item $\displaystyle D \sim \mathrm{Ber}(1 - U) + \mathrm{Poisson}(U)$, where
     $U$ is uniform on $\ClosedInterval{0, 1}$ and the Bernoulli and Poisson
     variables are independent conditional on $U$.
  \item $\displaystyle G_{D}(z) \defas \Expec{z^D}
     = \int_0^1\!\big(1 + (1 - x)(z - 1)\big)e^{x(z - 1)} dx
     = 2\,\frac{e^{z - 1} - 1}{z - 1} - 1$.
  \item For all $\displaystyle p \geq 1$,
    $\displaystyle\Expec*{\big}{D(D - 1)\cdots(D - p + 1)} = \frac{2}{p + 1}$.
  \item $\Prob{D = 0} = 1 - 2/e$ and, for $k \geq 1$,
    \[
      \Prob{D = k} \;=\; \frac{2}{e}\sum_{j > k} \frac{1}{j!} \, .
    \]
\end{mathlist}
\end{proposition}

\begin{proof}
First, for all $z \in \C\setminus\Set{1}$,
\[
  G_{D_n}(z) \;=\; 2\mleft(1 - \frac{1}{n}\mright)
    \frac{\mleft(1 + \frac{z - 1}{n - 1}\mright)^n - 1}{z\;-\;1} - 1
  \;\tendsto{n\to\infty}\;  2\,\frac{e^{z - 1} - 1}{z - 1} - 1.
\]
This pointwise convergence of the probability generating function of $D_n$
proves the convergence in distribution of $D_n$ to a random variable $D$
satisfying~(ii). Point~{(i)} then follows immediately from the integral
expression of $G_D$.

To compute the factorial moments of $D$, note that
\[
  G_D(z) \;=\; 2 \sum_{k \geq 0} \frac{(z - 1)^k}{(k + 1)!} \;-\; 1.
\]
As a result, for $p \geq 1$ the $p$-th derivative of $G_D$ is
\[
  G_D^{(p)}(z) \;=\; 2 \sum_{k \geq 0} \frac{(z - 1)^{k}}{(k + 1 + p)k!} \,,
\]
and, in particular, $\Expec{D(D - 1)\cdots(D - p + 1)} = G_D^{(p)}(1) =
\frac{2}{p + 1}$, proving (iii).

Finally, to prove (iv), using (i) we see that
\[
  \Prob{D = 0} \;=\; \int_0^1 x e^{-x} \, dx \;=\; 1 - \frac{2}{e}
\]
and that, for $k \geq 1$,
\[
  \Prob{D = k} \;=\; \frac{1}{k!} \int_0^1
      \mleft(k x^{k - 1} - k x^{k} + x^{k + 1}\mright) \,e^{-x}\, dx.
\]
Noting that $\big(k x^{k - 1} - k x^{k} + x^{k + 1}\big)e^{-x} = 2x^ke^{-x}+\frac{d}{dx}\mleft((x^{k}-x^{k+1})e^{-x}\mright)$, we get
\[
  \Prob{D = k}=\frac{2}{k!}\int_0^1x^ke^{-x}\,dx,
\] 
and an easy integration by parts yields
\[
  \Prob{D = k + 1} \;=\; \Prob{D = k} \;-\; \frac{2}{e(k + 1)!} \,, 
\]
from which (iv) follows by induction.
\end{proof}

Before closing this section, let us give an asymptotic equivalent of the tail
of $D_n$. We will need it in the proof of Theorem~\ref{thmMaxDegree} on the
largest degree.

\begin{proposition} \label{propTailDegree}
Let $D_n$ be the degree of a fixed vertex of $\MF{n}$ and let $D$ have the
asymptotic distribution of $D_n$. 
\begin{mathlist}
\item For all $k \geq 1$,
  \[
    \frac{2/e}{(k + 1)!} \;\leq\; \Prob{D \geq k} \;\leq\; 
    \mleft(1 + \frac{1}{k}\mright)^2 \frac{2/e}{(k + 1)!}\,.
  \]
\item For all $K_n = o(\sqrt{n})$, there exists $\epsilon_n = o(1)$ such
  that, for all $k \leq K_n$,
  \[
    \Abs{\Prob{D_n \geq k} - \Prob{D \geq k}}
    \;\leq\; \epsilon_n\, \Prob{D \geq k} \,.
  \]
\item For all $k_n \to + \infty$ and $K_n \geq k_n$ such that
  $K_n = o(\sqrt{n})$,
  \[
    \Prob{D_n \geq k} \;\sim\; \frac{2/e}{(k + 1)!}\,, 
  \]
  uniformly in $k$ such that $k_n \leq k \leq K_n$.
\end{mathlist}
\end{proposition}

\begin{proof}
First, observe that
\[
  \frac{1}{(\ell + 1)!} \;\leq\;
  \frac{1}{\ell \cdot \ell!} - \frac{1}{(\ell + 1)\cdot (\ell + 1)!}\,, 
\]
so that
\[
  \sum_{\ell > i} \frac{1}{\ell !} \;\leq\; \frac{1}{i\cdot i!}
  \;=\;\mleft(1 + \frac{1}{i}\mright) \frac{1}{(i + 1)!}\,.
\]
Recalling from Proposition~\ref{propLimitLawDegree} that
\[
  \Prob{D \geq k} \;=\;
  \frac{2}{e}\sum_{i \geq k} \sum_{\ell > i}\frac{1}{\ell!}\,,
\]
point (i) follows readily.

The proof of (ii) is somewhat technical so we only outline it here and refer
the reader to Section~\ref{appEndProofTailDegree} of the Appendix for the
detailed calculations.

Consider the function
\[
  \Delta_n(z) \;=\;
  \sum_{i \geq 0} \big(\Prob{D \geq i} - \Prob{D_n \geq i}\big) z^i.
\]
With this function, (ii) can be re-expressed as
\[
  \Delta_n^{(k)}(0) \;=\; \frac{\epsilon_n}{k + 1} \quad
  \text{for all } k \leq K_n = o(\sqrt{n}),
\]
where $\Delta_n^{(k)}$ denotes the $k$-th derivative of $\Delta_n$.
But $\Delta_n$ can be expressed in terms of the generating functions of
$D$ and $D_n$, namely as
\[
  \Delta_n(z) \;=\; \mleft(1 + \frac{1}{z - 1}\mright)
  \big(G_D(z) - G_{D_n}(z)\big).
\]
The expressions of $G_D$ and $G_{D_n}$ obtained in
Propositions~\ref{propLawDegree} and~\ref{propLimitLawDegree} thus make it
straightforward to obtain a power series expansion of $\Delta_n$ at $z = 1$,
and this expansion can be used to bound $\Delta_n^{(k)}(0)$ and conclude the
proof.

Finally, (iii) is a direct consequence of (i) and (ii). 
\end{proof}

\subsection{Largest degree} \label{secMaxDegree}

The aim of this section is to prove Theorem~\ref{thmMaxDegree} concerning
the largest degree of~$\MF{n}$, i.e.\ to show that
\[
  D_n^{\mathrm{max}} =\; \frac{\log n}{\log\log n} \;+\;
  \big(1 + o_\mathrm{p}(1)\big)\, \frac{\log n\,\log\log\log n}{(\log\log n)^{2}},
\]
where $o_\mathrm{p}(1)$ denotes a sequence of random variables
that goes to $0$ in probability and $D_n^{\mathrm{max}} = \max_{v}
D_n^{(v)}$.

Our proof is a standard application of the first and second moment method and
it implies that, asymptotically, $D_n^{\mathrm{max}}$ behaves like the maximum of
$n$ independent random variables distributed as~$D$.  As is typically the case
with this method, the main difficulty -- and therefore the bulk of the proof --
consists in bounding the asymptotic dependency between two fixed vertices
of~$\MF{n}$.

Because the first and second moment method part of our reasoning will also be
used in the proof of Theorem~\ref{thm:largestTree} concerning the size of the
largest tree, we isolate it as a lemma, whose proof we recall for the sake
of completeness.

\begin{lemma} \label{lemmaMomentsMethod}
For all integers $n$, let $(X_n^{(1)}, \ldots, X_n^{(n)})$ be a vector of
exchangeable random variables and
\[
  X_n^{\max} = \max\Set{X_n^{(i)} \suchthat i = 1, \ldots, n} \, .
\]
Write $p_n(k)$ for $\Prob*{\normalsize}{X_n^{(i)} \geq k}$, and suppose that
there exists a sequence $(m_n)$ and a constant~$\beta$ such that, for all
$\epsilon > 0$, as $n \to \infty$,
\begin{mathlist}
\item $n p_n((\beta + \epsilon) m_n) \to 0$.
\item $n p_n((\beta - \epsilon) m_n) \to +\infty$.
\item $\Prob*{\big}{X_n^{(1)} \geq (\beta - \epsilon)m_n,\,
  X_n^{(2)} \geq (\beta - \epsilon)m_n} \sim p_n((\beta - \epsilon)m_n)^2$.
\end{mathlist}
Then for all $\epsilon>0$,
\[
  \Prob*{\big}{X_n^{\max} \geq (\beta + \epsilon)\, m_n} \to 0
  \quad \text{and}\quad
  \Prob*{\big}{X_n^{\max} \geq (\beta - \epsilon) m_n} \to 1,
\]
which can also be written
\[
  X_n^{\max} = \big(\beta + o_{\mathrm{p}}(1)\big)\, m_n,
\]
where $o_{\mathrm{p}}(1)$ denotes a sequence of random variables that goes to
$0$ in probability.
\end{lemma}

\begin{proof}
First,
\begin{align*}
  \Prob*{\big}{X_n^{\max} \geq (\beta + \epsilon) m_n}
  = &\; \Prob{\bigcup_{i = 1}^n \Set{X_n^{(i)} \geq (\beta + \epsilon) m_n}} \\
  \leq &\; np_n\big((\beta + \epsilon) m_n\big) \,, 
\end{align*}
which goes to zero by~(i). Now, denote by
\[
  Z_n = \sum_{i = 1}^n \Indic*{\{X_n^{(i)} \geq\, (\beta - \epsilon) m_n\}}
\]
the number of variables $X_n^{(i)}$ that are greater than or equal to
$(\beta - \epsilon) m_n$.
Using the Cauchy--Schwartz inequality, we have $\Expec{Z_n}^2=\Expec{Z_n\Indic{Z_n>0}}^2 \leq\Expec{Z_n^2}\Prob{Z_n>0}$ so
\[
  \Prob*{\big}{X_n^{\max} \geq (\beta - \epsilon) m_n} \;=\;
  \Prob{Z_n > 0} \;\geq\; \frac{\Expec{Z_n}^2}{\Expec{Z_n^2}} \,.
\]
Moreover,
\begin{align*}
  \Expec{Z_n^2} \;&=\; n p_n\big((\beta - \epsilon) m_n\big) \\
  &+\; n(n - 1) \Prob{X_n^{(1)} \geq (\beta - \epsilon) m_n, \,
  X_n^{(2)} \geq (\beta - \epsilon) m_n} \,, 
\end{align*}
and so, by (ii) and (iii), $\Expec{Z_n}^2 / \Expec{Z_n^2} \to 1$ as
$n \to \infty$.
\end{proof}

\begin{remark} \label{rem:three-prime}
Note that under assumption (ii) of this lemma, for any $\epsilon>0$,
letting $n\to\infty$ in $\Expec{Z_n}^2 / \Expec{Z_n^2} \leq 1$ shows that
\[
  p_n((\beta-\epsilon)m_n)^2 \;\leq\;
    \Prob{X_n^{(1)} \geq (\beta - \epsilon) m_n,\,
          X_n^{(2)} \geq (\beta - \epsilon) m_n}(1+o(1)).
\]
Therefore, to prove (iii) it suffices to show
\begin{itemize}
  \item[(iii')] $\displaystyle \Prob{X_n^{(1)} \geq (\beta - \epsilon) m_n,\,
                X_n^{(2)} \geq (\beta - \epsilon) m_n} \;\leq\; p_n((\beta-\epsilon)m_n)^2(1+o(1)).$ \qedhere
\end{itemize}
\end{remark}
\bigskip

We now turn to the proof of Theorem~\ref{thmMaxDegree}.

\begin{proof}[Proof of Theorem~\ref{thmMaxDegree}]

\change{Instead of proving the theorem directly for the degree random variables
$(D_n^{(1)}, \ldots, D_n^{(n)})$, we prove it for the out-degrees, which we denote by $(\tilde{D}_n^{(1)}, \ldots, \tilde{D}_n^{(n)})$ and whose maximum has the same asymptotic behavior as $D_n^{\max}$.  The point in doing this is that
the tails of the out-degrees are less correlated than
those of the variables~$D_n^{(v)}$, making it easier to study their maximum
by the first and second moment method.}
  
Remember from Section~\ref{secDegreeFixed} that, in the UA construction,
\[
  D_n^{(v)} =\;
  I_{B_v} \;+ \sum_{\ell = B_v + 1}^n \!\!\! X^{(v)}_\ell, 
\]
where $B_v$ is the step at which vertex~$v$ was added,
$X^{(v)}_\ell$ is the indicator of ``the $\ell$-th vertex is linked to
vertex~$v$'', and $I_\ell$ is the indicator of ``the $\ell$-th vertex is
linked to a previously added vertex''. With this notation, let
\[
  \tilde{D}_n^{(v)} =
  \sum_{\ell = B_v + 1}^n \!\!\! X^{(v)}_\ell,
\]
denote the out-degree of vertex $v$,
and set $\tilde{D}_n^{\max} = \max\Set*{\tilde{D}_n^{(v)} : v = 1, \ldots, n}$.
Since $\tilde{D}_n^{\max}$ and $D_n^{\max}$ differ by at most~1, for any 
$m_n \to +\infty$,
\[
  D_n^{\max} -\, \tilde{D}_n^{\max} \;=\; o_\mathrm{p}(m_n) \, ,
\]
i.e.\ $(D_n^{\max} -\, \tilde{D}_n^{\max})/m_n$ goes to 0 in probability. Thus, to 
prove the theorem we apply Lemma~\ref{lemmaMomentsMethod} to the variables
\[
  \mleft(\tilde{D}_n^{(1)} - \frac{\log n}{\log\log n}, \ldots,\;
  \tilde{D}_n^{(n)} - \frac{\log n}{\log\log n}\mright)\,, 
\]
with $m_n = (\log n) (\log\log\log n) / (\log\log n)^2$ and $\beta = 1$.

Using Proposition~\ref{propTailDegree} and Stirling's formula, we see that
for any $k_n = o(\sqrt{n})$,
\[
  \log\big(\Prob{D_n \geq k_n}\big) \;=\;
   - k_n \log k_n \;+\; k_n \;+\; O(\log k_n)\,.
\]
Writing $\tilde{D}_n$ to refer to the common distribution of the
variables~$\tilde{D}_n^{(v)}$, since
\[
  \Prob*{\big}{D_n \geq k_n + 1} \;\leq\;
  \Prob{\tilde{D}_n \geq k_n} \;\leq\;
  \Prob*{\big}{D_n \geq k_n} \,,
\]
we also have
\[
  \log\big(\Prob{\tilde{D}_n \geq k_n}\big) \;=\;
   - k_n \log k_n \;+\; k_n \;+\; O(\log k_n)\,.
\]
In particular, for $k_n = (\log n) / (\log\log n) + \gamma \, m_n$ with
\[
  m_n = \frac{\log n \, \log\log\log n}{(\log\log n)^2} \,, 
\]
this gives
\begin{equation} \label{eqLogTailDtilde}
  \log\big(\Prob{\tilde{D}_n \geq k_n}\big) \;=\; - \log n \;-\;
  \big(\gamma - 1\big)\, \frac{\log n \;  \log\log\log n}{\log\log n} \;+\;
  O\mleft(\frac{\log n}{\log\log n}\mright)
\end{equation}

\pagebreak 

As a result, for all $\epsilon > 0$,
\begin{mathlist}
  \item $n\,\Prob*{\Big}{\tilde{D}_n - \dfrac{\log n}{\log\log n} \geq (1 + \epsilon) m_n} \to 0$.
  \item $n\,\Prob*{\Big}{\tilde{D}_n - \dfrac{\log n}{\log\log n} \geq (1 - \epsilon) m_n} \to +\infty$.
\end{mathlist}
Thus, to apply Lemma~\ref{lemmaMomentsMethod} and
finish the proof it suffices to show that
\[
  \Prob{\tilde{D}_n^{(1)} \geq k_n,\,
        \tilde{D}_n^{(2)} \geq k_n} \;\sim\;
  \Prob{\tilde{D}_n \geq k_n}^2
\]
whenever $k_n = (\log n) / (\log\log n) + (1 - \epsilon) m_n$.
More precisely, using Remark~\ref{rem:three-prime} it is sufficient to show that
\[
  \Prob{\tilde{D}_n^{(1)} \geq k_n,\,
    \tilde{D}_n^{(2)} \geq k_n} \;\leq\;
  \Prob{\tilde{D}_n \geq k_n}^2 \,+\, o\Big(\Prob{\tilde{D}_n \geq k_n}^2\Big).
\]
First let us fix $b_1\neq b_2\in\{1,\ldots,n\}$.
Conditional on $\{B_1=b_1,B_2=b_2\}$, recall that the variables $(X^{(2)}_\ell,\, b_2+1\leq \ell\leq n)$ are independent Bernoulli variables with parameter $1/(n - 1)$.
By further conditioning on the variables $X^{(1)}_\ell$, the independence of $(X^{(2)}_\ell,\, b_2+1\leq \ell\leq n)$ still holds but their distribution is changed.
Indeed, choose $(x_\ell, \,\ell \neq b_1)\in\{0,1\}^{n-1}$ and consider the event
\[
  A \defas \Set{B_1=b_1,B_2=b_2,\; \forall\ell\neq b_1, X^{(1)}_{\ell} = x_{\ell}}.
\]
Then by construction, for all $\ell\notin\{b_1,b_2\}$, we have
  \[
  \Prob{X^{(2)}_\ell=1 \given A} = \begin{cases}
  0 &\text{ if } x_\ell = 1\\
  \frac{1}{n-2} &\text{ if } x_\ell = 0.
  \end{cases}
\]
Consequently $X^{(2)}_\ell$ is always stochastically dominated by a Bernoulli($\frac{1}{n-2}$) random variable, and so we bound the distribution of $\tilde{D}_n^{(2)}=\sum_{\ell>b_2}X^{(2)}_\ell$ conditional on~$A$~by
\[
  \big(\tilde{D}_n^{(2)} \;\big|\; A\big) \;\overset{d}{\leq}\; \text{Binomial}\left(n-b_2,\frac{1}{n-2}\right).
\]
To get a bound on the distribution of $\tilde{D}_n^{(2)}$ conditional on $\tilde{D}_n^{(1)}=i$ for some $i$, first note that summing over all configurations $b_1,b_2,(x_\ell,\ell\neq b_1)$ such that $\sum_{\ell>b_1} x_\ell=i$ gives
\[
  \big(\tilde{D}_n^{(2)} \;\big|\; B_1=b_1,B_2=b_2,\,\tilde{D}_n^{(1)}=i\big) \;\overset{d}{\leq}\; \text{Binomial}\left(n-b_2,\frac{1}{n-2}\right).
\]
Let us now write for conciseness $L_1=n-B_1$ and $L_2=n-B_2$.
Note that $L_2$ is not independent of $\Set*{\tilde{D}_n^{(1)} = i}$
because they are linked by $L_1$. Indeed, $L_1$ is positively
correlated to $\tilde{D}_n^{(1)}$ and we always have $L_2 \neq L_1$.
Nevertheless, since conditional on $L_1$, $L_2$ is independent of
$\tilde{D}^{(1)}_n$ and uniform on $\Set{0, \ldots, n - 1} \setminus L_1$, we
have the following stochastic ordering:
\[
  \big(L_2 \;\big|\; B_1=b_1,\,\tilde{D}_n^{(1)}=i\big)
  \;\;\overset{d}{\leq}\;\;
  \altoverline{L}{2}\,, 
\]
where $\altoverline{L}{2}$ is uniformly distributed on $\Set{1, \ldots, n - 1}$.
Summing over $b_1$ and $b_2$, we thus get
\[
  \big(\tilde{D}_n^{(2)} \;\big|\;\tilde{D}_n^{(1)}=i\big) \;\overset{d}{\leq}\; \text{Binomial}\left(\altoverline{L}{2},\frac{1}{n-2}\right).
\]
Let us define a random variable $M_n\sim\text{Bin}\left(\altoverline{L}{2},\frac{1}{n-2}\right)$.
As the previous bound is uniform in $i$, we have
\[
  \Prob{\tilde{D}_n^{(2)} \geq k_n\,
    \given \tilde{D}_n^{(1)} \geq k_n} \;\leq\;
  \Prob{M_n \geq k_n}.
\]
To conclude, it is sufficient to show that $\Prob{M_n \geq k_n} \sim \Prob{\tilde{D}_n \geq k_n}$ since this implies 
\[
  \Prob{\tilde{D}_n^{(1)} \geq k_n,\,\tilde{D}_n^{(2)} \geq k_n} \;\leq\;
  \Prob{\tilde{D}_n \geq k_n}\Prob{M_n \geq k_n} \sim \Prob{\tilde{D}_n \geq k_n}^2.
\]
For this, define on the same probability space as the variables $\altoverline{L}{2}$ and $M_n$ the variable
\[
  \altunderline{L}{2} \defas \altoverline{L}{2}\Indic{\altoverline{L}{2}\leq n-2}.
\]
$\altunderline{L}{2}$ is then uniformly distributed on $\Set{0,\ldots,n-2}$, and we have the equality in distribution
\[
  M_n\Indic{\altoverline{L}{2}\leq n-2} \;\overset{d}{=}\; \tilde{D}_{n-1} \;\sim\; \text{Binomial}\left(\altunderline{L}{2},\frac{1}{n-2}\right).
\]
As the two variables $M_n$ and $M_n\Indic{\altoverline{L}{2}\leq n-2}$ differ on an event of probability no greater than $1/(n-1)$, we have
\[
  \Prob{M_n \geq k_n} = \Prob{\tilde{D}_{n-1} \geq k_n} +O\big(\tfrac{1}{n}\big),
\]
and finally \eqref{eqLogTailDtilde} with $\gamma=(1-\epsilon)$ allows us to conclude that this expression is indeed equivalent to $\Prob{\tilde{D}_{n} \geq k_n}$.
\end{proof}

\section{Tree sizes} \label{secTreeSizes}

In this section, we study the size of the trees composing the Moran forest.
Section~\ref{secSizeUnifTree} is concerned with the typical size of these
trees, while Section~\ref{secMaxTree} focuses on the asymptotics of the size
of the largest tree. But before going any further we need to introduce a
process that will play a central role throughout the rest of this paper.

\subsection{A discrete-time Yule process} \label{secUpsilon}

Let $\Upsilon_n = (\Upsilon_n(\ell),\, \ell \ge 0)$ be the pure birth
Markov chain defined by $\Upsilon_n(0) = 1$ and the following transition
probabilities:
\[
  \Prob*{\big}{\Upsilon_n(\ell+1) = j \given \Upsilon_n(\ell) = i} = 
  \begin{cases}
    \frac{i}{n-1} &\text{ if $j = i+1$}\\
    1-\frac{i}{n-1} &\text{ if $j = i$},
  \end{cases}
\]
and stopped when reaching $n$.

The reason why this process will play an important role when studying the
trees of $\MF{n}$ is the following: let $\Ti^{(v)}_n$ denote the tree
containing $v$, and $\tilde{\Ti}^{(v)}_n$ the subtree descending from $v$ in
the UA construction -- that is, letting $m(v)$ denote the mother of $v$ and
$\Ti^{(v)}_n \setminus \{vm(v)\}$ the forest obtained by removing the edge
between $v$ and $m(v)$ from $\Ti^{(v)}_n$ (if that edge existed),
$\tilde{\Ti}^{(v)}_n$ is the tree of $\Ti^{(v)}_n \setminus \{vm(v)\}$
containing~$v$. Recalling that $L_v$ denotes the number of steps after
vertex~$v$ was added in the UA construction and letting
$\tilde{T}^{(v)}_n = |\tilde{\Ti}^{(v)}_n|$ be the size of $\tilde{\Ti}^{(v)}_n$,
we have
\[
  \tilde{T}^{(v)}_n \;\overset{d}{=}\; \Upsilon_n(L_v)\,, 
\]
where $\Upsilon_n$ is independent of $L_v$. In particular, the size
of a tree created at step~${n - h}$ of the UA construction is
distributed as $\Upsilon_n(h)$.

In the rest of this section, we list a few basic properties of~$\Upsilon_n$
that will be used in subsequent proofs.

\begin{lemma} \label{lemmaDiscreteYuleExpect}
For all $0 \le \ell \le n-1$,
\[
  \Expec{\Upsilon_n(\ell)} \;=\; \mleft(1 + \frac{1}{n-1} \mright)^{\ell}\,.
\]
\end{lemma}

\begin{proof}
For $0 \le \ell < n-1$, we have $\Upsilon_n(\ell) < n$ almost surely, therefore we can write
\begin{align*}
  \Expec*{\big}{\Upsilon_n(\ell+1) \given \Upsilon_n(\ell)} \;
  &=\; \frac{\Upsilon_n(\ell)}{n-1}\big(\Upsilon_n(\ell) + 1\big) \;+\;
    \mleft(1-\frac{\Upsilon_n(\ell)}{n-1}\mright)\Upsilon_n(\ell) \\
  &=\; \Upsilon_n(\ell)\mleft(1+\frac{1}{n-1}\mright)\,, 
\end{align*}
and the result follows by induction.
\end{proof}

We now compare the discrete-time process $\Upsilon_n$ to the 
Yule process. By Yule process, we refer to the continuous-time
Markov chain $(Y(t),\,t\geq 0)$ that jumps from $i$ to $i + 1$ at rate~$i$
(see e.g.~\cite{Ross1995}, Section~5.3).

\begin{lemma} \label{lemmaYuleConvergence}
As $n \to \infty$,
\[
  \big(\Upsilon_n(\Floor{tn}),\, t \ge 0 \big) \implies \big(Y(t),\, t \ge 0\big),
\]
where ``$\implies$'' denotes convergence in distribution in the Skorokhod
space~\cite{Billingsley1999}, and ${(Y(t),\, t \ge 0)}$ is a Yule process.
\end{lemma}

\begin{proof}
Since both processes only have increments of +1, it suffices to prove that the
sequence of jump times of $(\Upsilon_n(\Floor{tn}),\, t \ge 0)$
converges in distribution to that of the Yule process. For $1 \le i \le n$, let
\[
  t_n(i) = \inf \Set*[\big]{\ell \ge 0 \suchthat \Upsilon_n(\ell) = i}
\]
be the jump times of the chain $\Upsilon_n$. By the strong Markov
property, the variables $(t_n(i+1) - t_n(i),\, 1 \le i \le n-1)$
are independent, and $t_n(i+1) - t_n(i) \sim \mathrm{Geometric}(\frac{i}{n-1})$.
Therefore,
\[
  \Big(\tfrac{1}{n}\big(t_n(i+1) - t_n(i)\big),\, 1 \le i \le n-1\Big)
  \;\tendsto[d]{n \to \infty}\;
  (\mathscr{E}(i),\, i \ge 1),
\]
where the variables $(\mathscr{E}(i),\, i \ge 1)$ are independent and
$\mathscr{E}(i) \sim \mathrm{Exponential}(i)$. This concludes the proof.
\end{proof}

\begin{lemma} \label{lem:bounds-on-upsilon}
For all integers $0\leq k\leq \ell \leq n-1$,
\[
  \Prob{Y\big(\tfrac{\ell-k+1}{n-1}\big) > k} \;\leq\;
  \Prob{\Upsilon_n(\ell) > k} \;\leq\;
  \Prob{Y\big(\lambda_n(k)\tfrac{\ell}{n-1}\big) > k},
\]
where
\[
  \lambda_n(k) = -\frac{n-1}{k} \log\mleft(1-\frac{k}{n-1}\mright).
\]
\end{lemma}

\begin{proof}
Let us start with the upper bound, and write $\lambda\defas\lambda_n(k)$ for
simplicity. Note that, for all $t\geq 0$ and $i\geq 1$,
\[
  \Prob{Y\big(t+\tfrac{\lambda}{n-1}\big) = i \given Y(t) = i}
  = e^{-\frac{i\lambda}{n-1}},
\]
and that we have chosen $\lambda$ such that if $i\leq k$ then
\[
  e^{-\frac{i\lambda}{n-1}} \;\leq\;
  1 - \frac{i}{n-1} \;=\;
  \Prob*\big{\Upsilon_n(\ell+1) = i \given \Upsilon_n(\ell) = i}.
\]
Thus, until it reaches $k+1$ individuals, the process $\Upsilon_n$ is dominated
by the Markov chain $(Y(\frac{\lambda\ell}{n-1}),\, 0\leq \ell \leq n-1)$.
This shows that
\[
  \Prob{\Upsilon_n(\ell) > k} \;\leq\;
  \Prob{Y\big(\tfrac{\lambda\ell}{n-1}\big) > k}\,, 
\]
proving the second inequality of the lemma.

To prove the first inequality, we couple $\Upsilon_n$ with a ``censored''
Yule process~$Y_{\mathrm{c}}$. Intuitively, this censoring
consists in ignoring births that occur less than $1/(n - 1)$ unit of time
after another birth.

Formally, we define $Y_{\mathrm{c}}$ by specifying the sequence
$t_0 = 0 < t_1 < t_2 < \ldots$ of times corresponding to births in the population.
Let $(\mathscr{E}_i,\,i\geq 1)$ be an independent sequence of exponential random variables
where $\mathscr{E}_i \sim \mathrm{Exponential}(i)$. Set $t_0 = 0$ and, for each $i\geq 1$,
\begin{equation}
  t_i \;\defas\; \mathscr{E}_1+\sum_{j=2}^{i}\left(\frac{1}{n-1}+\mathscr{E}_j\right)
  \;=\; \frac{i-1}{n-1}+\sum_{j=1}^{i}\mathscr{E}_j.
\end{equation}
We now define, for all $t \geq 0$, 
\[
  Y_{\mathrm{c}}(t) \;\defas\; 1+\sum_{i\geq 1}\Indic{t_i \leq t}
  \;=\; \sum_{i\geq 1} i\,\Indic{t_{i-1}\leq t < t_i}.
\]

The censoring of the Yule process after birth events implies that for any time
$t\geq 0$, the random variable
$Y_{\mathrm{c}}(t+\frac{1}{n-1}) - Y_{\mathrm{c}}(t)$ takes values in $\{0,1\}$.
Furthermore, for any $i\in\N$,
\[
  \Prob*{\big}{Y_{\mathrm{c}}(t+\tfrac{1}{n-1}) = i+1 \given Y_{\mathrm{c}}(t) = i}
  \;\leq\; 1-e^{-\frac{i}{n-1}} \;\leq\; \frac{i}{n-1}.
\]
Therefore, we can couple $(\Upsilon_n(\ell),\,0\leq \ell\leq n-1)$ and
$(Y_{\mathrm{c}}(t),\,t\geq 0)$ in such a way that, for all $0\leq \ell \leq n-1$,
\[
  Y_{\mathrm{c}}\mleft(\tfrac{\ell}{n-1}\mright) \;\leq\; \Upsilon_n(\ell).
\]
Now, by construction, the sequence $(t_i - \frac{i-1}{n-1},\,i\geq 1)$ has the
distribution of the sequence of jump times of a Yule process.
Therefore,
\begin{align*}
  \Prob{\Upsilon_n(\ell) > k}
  \;&\geq\; \Prob{Y_{\mathrm{c}}\big(\tfrac{\ell}{n-1}\big) > k}\\
  \;&=\;    \Prob{t_k \le \tfrac{\ell}{n-1}}\\
  \;&=\;    \Prob{t_k-\tfrac{k-1}{n-1} \le \tfrac{\ell-k+1}{n-1}}\\
  \;&=\;    \Prob{Y\big(\tfrac{\ell-k+1}{n-1}\big) > k},
\end{align*}
which yields the lower bound of the lemma.
\end{proof}

\change{We now use the previous lemma to obtain the following result, which will be used to derive asymptotics for the tail probability of the size of a tree in the Moran forest.

\begin{proposition} \label{prop:tail-tree}
  Let $L$ be a uniform random variable on $\{0,\dots,n-1\}$, independent of the process $\Upsilon_n$.
  Then for any sequence of integers $k_n\to\infty$ with $k_n=o(\sqrt{n})$,
  \[
  \Prob{\Upsilon_n(L) > k_n} \sim \frac{e}{k_n}(1-e^{-1})^{k_n+1}.
  \]
\end{proposition}
}
\begin{proof}
  Using the upper bound in Lemma \ref{lem:bounds-on-upsilon}, we have
  \begin{align*}
  \Prob{\Upsilon_n(L)>{k_n}}
  \;&\leq\; \frac{1}{n}\sum_{\ell=0}^{n-1}
  \Prob{Y\big(\lambda_n(k)\tfrac{\ell}{n-1}\big)>k_n}\\
  &=\; \frac{n-1}{n}\int_{0}^{1}
  \Prob{Y\big(\lambda_n(k_n)\tfrac{\Floor{x(n-1)}}{n-1}\big)>k_n} \,dx
  \;+\; \frac{1}{n}\Prob{Y(\lambda_n(k_n)) > k_n}\\
  &\leq\; \int_{0}^{1} \Prob{Y(\lambda_n(k_n) x)>k_n} \,dx 
  \;+\; \frac{1}{n} (1-e^{-\lambda_n(k_n)})^{k_n}\\
  &=\; \int_{0}^{1} \big(1-e^{-\lambda_n(k_n) x}\big)^{k_n} dx
  \;+\; \frac{1}{n}(1-e^{-\lambda_n(k_n)})^{k_n}.
  \end{align*}
  Now recall that $\lambda_n(k_n) = -\frac{n-1}{{k_n}} \log\big(1-\frac{k_n}{n-1}\big) = 1 + O(\frac{k_n}{n})$, so uniformly in $x\in\ClosedInterval{0,1}$,
  \[
  e^{-\lambda_n(k_n) x} = e^{-x} + O\big(\frac{k_n}{n}\big).
  \]
  Since $k_n=o(\sqrt{n})$, we have $k_n/n=o(1/k_n)$ and thus Lemma~\ref{lem:integral-equivalent}
  from the Appendix gives
  \[
  \int_{0}^{1} \big(1-e^{-\lambda_n(k_n) x}\big)^{k_n} \,dx \sim \frac{e}{k_n}(1-e^{-1})^{k_n+1}.
  \]
  Elementary calculations show that when $k_n=o(\sqrt{n})$, we also have
  \[
  \frac{1}{n}(1-e^{-\lambda_n(k_n)})^{k_n} \sim \frac{1}{n}(1-e^{-1})^{k_n} = o\Big(\frac{(1-e^{-1})^{k_n}}{k_n}\Big).
  \]
  It remains to examine the lower bound in Lemma~\ref{lem:bounds-on-upsilon}.
  As above, we get an integral
  \begin{align*}
  \Prob{\Upsilon_n(L)>k_n} & \geq \frac{1}{n}\sum_{\ell=0}^{n-1} \Prob{Y\big(\tfrac{\ell-k_n+1}{n-1}\big)>k_n}\\
  &\geq \frac{n-1}{n}\int_{0}^{1}\Prob{Y\big(\tfrac{\Ceil{x(n-1)}-k_n}{n-1}\big)>k_n} \,dx\\
  &\geq \frac{n-1}{n}\int_{0}^{1}\Prob{Y\big(x-\tfrac{k_n}{n-1}\big)>k_n} \,dx.
  \end{align*}
  Since
  \[
  \Prob{Y\big(x-\tfrac{k_n}{n-1}\big)>k_n} = \big(1-\exp(-x+\tfrac{k_n}{n-1})\big)^{k_n} = \big(1-e^{-x}+O(k_n/n)\big)^{k_n},
  \]
  using Lemma \ref{lem:integral-equivalent} again, we get
  \[
  \Prob{\Upsilon_n(L)>k_n} \geq \frac{n-1}{n}\int_{0}^{1} \Prob{Y\big(x-\tfrac{k_n}{n-1}\big)>k_n} \, dx \sim \frac{e}{k_n}(1-e^{-1})^{k_n+1},
  \]
  which completes the proof.
\end{proof}

\subsection{Size of some random trees} \label{secSizeUnifTree}

In this section, we study the size of some typical trees
of $\MF{n}$. In particular, we study the asymptotics of the
size $T_n^{(1)}$ of the tree containing vertex~1 and of the size $T^U_n$ of
a tree sampled uniformly at random among the trees composing $\MF{n}$.
Our main result is the following theorem.

\begin{theorem} \label{thmAsymptoticTreeSize}
~
\begin{mathlist}
  \item Let $T^U_n$ be the size of a uniform tree of $\MF{n}$. Then,
    \[
      \Prob{T^U_n = k} \;\tendsto{n \to \infty}\;
      2 \int_0^1 x e^{-x}(1-e^{-x})^{k-1} dx,
    \]
    that is, ${T^U_n} \tendsto[d]{} T^U$ where $T^U \sim
    \mathrm{Geometric}(e^{-X})$, and $X \sim 2xdx$ on
    $\ClosedInterval{0, 1}$.
  \item Let $T^{(1)}_n$ be the size of the tree containing vertex~1 in $\MF{n}$.
    Then,
    \[
      \Prob{T_n^{(1)} = k} \tendsto{n \to \infty} k \int_0^1 x
      e^{-x}(1-e^{-x})^{k-1} dx,
    \]
    that is, $T_n^{(1)}$ converges in distribution to the size-biasing of
    $T^U$.
\end{mathlist}
\end{theorem}

\begin{remark} \label{remAsympSizeBias}
Note that even though the limit distribution of $T_n^{(1)}$ is the size-biased
limit distribution of $T^U_n$, for finite $n$ the distribution of $T_n^{(1)}$
is \emph{not} the size-biased distribution of~$T^U_n$.
\change{
Indeed, note that
\[
  \Prob{T_n^{(1)}=k}
  \;=\; \Expec*{\Big}{\sum_{\mathcal{T}\in \MF{n}}\Indic{\Abs{\mathcal{T}}=k}\Indic{1\in \mathcal{T}}}
  \;=\; \frac{k}{n}\,\Expec*{\Big}{\sum_{\mathcal{T}\in \MF{n}}\Indic{\Abs{\mathcal{T}}=k}},
\]
while
\[
  \Prob{T_n^{U}=k} \;=\; \Expec*{\Big}{\frac{1}{N_n}\sum_{\mathcal{T}\in \MF{n}}\Indic{\Abs{\mathcal{T}}=k}},
\]
where, as in Section \ref{secNumberTrees}, $N_n$ denotes the number of trees in $\MF{n}$.
However, these computations are enough to show that, in the limit, the size-biasing holds:
indeed, note that $n/N_n\to 2$ in probability by Proposition \ref{propNumberTrees}.
Furthermore, using for instance Hoeffding's inequality \cite{Hoeffding1963} to
control the deviation of $N_n$ from its mean, it is easy to show that
$n/N_n\to 2$ in $L^1$ as well, so that
\[
  \Abs{\Prob{T_n^{(1)}=k}-\tfrac{k}{2}\Prob{T_n^{U}=k}} \;\leq\; \frac{k}{2}\, \Expec{\Abs{\frac{n}{N_n}-2}} \to 0.
\]
This shows that points (i) and (ii) of Theorem~\ref{thmAsymptoticTreeSize} are equivalent.
}
\end{remark}

We start by giving the distribution of $T_n^{(1)}$ in terms of the
process~$\Upsilon_n$ defined in Section~\ref{secUpsilon}. For this, we first
need to introduce some notation.  Let $\mathscr{T}^{(v)}_n$ be the tree
containing vertex~$v$ in $\MF{n}$. We denote by $H^{(v)}_n$ the number of steps
after the root of $\mathscr{T}^{(v)}_n$ was added in the UA construction.
Recalling the notation from Section~\ref{secUAC}, where
$\sigma^{-1}(v) \in \Set{1, \ldots, n}$ denotes the step of the UA construction
at which vertex $v$ was added, we thus have
\[
  H^{(v)}_n \;=\;
  n - \min\Set{\sigma^{-1}(u) \suchthat u \in \mathscr{T}^{(v)}_n} \,.
\]

\begin{proposition} \label{propBiasedDistribution} ~
Let \,$T^{(1)}_n$ be the size of the tree containing vertex~$1$ in $\MF{n}$,
and denote by $H^{(1)}_n$ the number of steps after the root of that tree
was added in the UA construction. Then,
\begin{mathlist}
\item For $0 \le h \le n-1$, $\displaystyle \Prob*{}{H_n^{(1)} = h} = \frac{h}{n(n-1)}
  \mleft(1+\frac{1}{n-1}\mright)^h$.
\item Conditional on $\Set*{H_n^{(1)} = h}$, $T_n^{(1)}$ is distributed
  as the size-biasing of $\Upsilon_n(h)$.
\end{mathlist}
\end{proposition}

\begin{remark} \label{remSizeBias}
  The size-biasing of $\Upsilon_n(h)$ can be easily represented as follows.
  Consider the Markov chain $\Upsilon^*_n = (\Upsilon^*_n(\ell),\, 0 \le \ell \le n-1)$ defined by $\Upsilon^*_n(0) = 1$ and the following transition
  probabilities:
  \[
  \Prob*{\big}{\Upsilon^*_n(\ell+1) = j \given \Upsilon^*_n(\ell) = i} = 
  \begin{cases}
  \frac{i+1}{n} &\text{ if $j = i+1$}\\
  1-\frac{i+1}{n} &\text{ if $j = i$}.
  \end{cases}
  \]
  A straightforward induction on $\ell$ shows that $\Upsilon^*_n(\ell)$ is distributed as the size-biasing of $\Upsilon_n(\ell)$.
\end{remark}

\begin{proof}
First, note that $H_n^{(1)} = h$ if and only if a
new tree is created at step ${n - h}$, and vertex~1 belongs to this tree.
Now, the probability that a new tree is created at step $n - h$ is 
$\frac{h}{n-1}$, and the size of this tree is then distributed as
$\Upsilon_n(h)$. Moreover, at the end of the UA construction the labels are
assigned to the vertices uniformly. As a result, conditional on a tree having
size~$i$, the probability that it contains vertex~1 is $i/n$. We thus have
\[
  \Prob{H_n^{(1)}= h,\, T_n^{(1)} = i} \;=\; \frac{h}{n-1}\cdot \frac{i}{n}\,
  \Prob{\Upsilon_n(h) = i}.
\]
Summing over $i$ and using Lemma~\ref{lemmaDiscreteYuleExpect} yields
\[
  \Prob{H_n^{(1)} = h} \;=\; \frac{h}{n(n-1)}\mleft(1+\frac{1}{n-1}\mright)^{h}.
\]
Finally,
\[
  \Prob{T_n^{(1)} = i \given H_n^{(1)} = h} = i\, \Prob*{\big}{\Upsilon_n(h) = i}
  \mleft(1+\frac{1}{n-1}\mright)^{-h},
\]
which concludes the proof.
\end{proof}

We can now turn to the proof of our main result.

\begin{proof}[Proof of Theorem~\ref{thmAsymptoticTreeSize}]
  \change{
  By Remark \ref{remSizeBias}, it is sufficient to prove (ii).
  Now from Proposition~\ref{propBiasedDistribution}, we can write
  \begin{align*}
    \Prob{T_n^{(1)}=k} &= 
    \frac{1}{n} \sum_{h = 0}^{n-1} \frac{h}{n-1}
    \Expec{\Upsilon_n(h) \Indic{\Upsilon_n(h)=k}}\\
    &= \frac{k}{n} \sum_{h = 0}^{n-1} \frac{h}{n-1} \Prob{\Upsilon_n(h)=k}.
  \end{align*}
  Therefore using Lemma~\ref{lemmaYuleConvergence}, we get
  \[
    \frac{k}{n} \sum_{h = 0}^{n-1} \frac{h}{n-1}\Prob{\Upsilon_n(h)=k}\tendsto{n \to \infty} k\int_0^1 x \Prob{Y(x)=k} dx,
  \]
  and recalling the well-known fact that $Y(x)$ has a $\textrm{Geometric}(e^{-x})$
  distribution (see for instance Section~5.3 in \cite{Ross1995}) concludes the proof.
  }
\end{proof}

\begin{remark}
  \change{
  If $H^U_n$ denotes the number of steps in the UA construction after
  the root of a uniformly chosen tree was added, one could give an
  alternative proof of Theorem~\ref{thmAsymptoticTreeSize} by
  showing that $H^U_n / n \to X$, and then using
  Lemma~\ref{lemmaYuleConvergence}.
  }
\end{remark}

\subsection{Size of the largest tree} \label{secMaxTree}

\change{
The goal of this section is to derive asymptotics for
$\Tmax_n\defas\max_{v}T_n^{(v)}$, the size of the largest tree in the Moran
forest on $n$ vertices, when $n\to\infty$. Namely, we show that
\[
  \Tmax_n = \alpha\big(\log n - (1+o_{\mathrm{p}}(1))\log\log n\big),
\]
where $\alpha = (1-\log(e-1))^{-1}$. Similarly to 
Theorem~\ref{thmMaxDegree} concerning the largest degree, this corresponds to
the maximum of $n/2$ independent $T^U_n$-distributed trees. Again the key
element to the proof is to control the asymptotic independence of two
distinct trees of $\MF{n}$.
}

As in Section~\ref{secUpsilon}, for any vertex $v$ let us define
$\tilde{\Ti}_n^{(v)}\subset \Ti_n^{(v)}$ as the subtree descending from $v$ in
the UA construction. For our purpose, it will be sufficient to study the size
$\tilde{T}_n^{(v)}\defas |\tilde{\Ti}_n^{(v)}|$ of those subtrees instead of
that of the trees $\Ti_n^{(v)}$. Indeed, observe that
\[
  \Tmax_n = \max_{v}\tilde{T}_n^{(v)},
\]
so that applying Lemma \ref{lemmaMomentsMethod} 
with $m_n=\alpha\log\log n$ and $\beta=-1$ to the exchangeable variables
$(\tilde{T}_n^{(1)}-\alpha\log n, \ldots, \tilde{T}_n^{(n)}-\alpha\log n)$
will prove the theorem.
Again, we omit the superscript and denote by $\tilde{T}_n$ a random variable
with distribution equal to that of $\tilde{T}_n^{(1)}$.

For the rest of the section, we thus study the tail probabilities of the
variable $\tilde{T}_n$.  Recall from the UA construction that the number $L$ of
steps after a fixed vertex was added is uniformly distributed on
$\Set{0,\ldots,n-1}$, and from Section~\ref{secUpsilon} that, conditional on
$\Set{L =\ell}$,
\[
  \tilde{T}_n \overset{d}{=} \Upsilon_n(\ell)\,.
\]
\change{As a consequence, applying directly Proposition \ref{prop:tail-tree} yields that for any sequence of integers $k_n\to\infty$ with $k_n=o(\sqrt{n})$,
\begin{equation} \label{eq:tail-tree}
  \Prob{\tilde{T}_n > k_n} \sim \frac{e}{k_n}(1-e^{-1})^{k_n+1}.
\end{equation}
}
Note that if $k_n$ is not integer-valued, then
\[
  \Prob{\tilde{T}_n > k_n}=\Prob{\tilde{T}_n > \Floor{k_n}}\sim \frac{e}{k_n}(1-e^{-1})^{\Floor{k_n}+1},
\]
which is not necessarily equivalent to $\frac{e}{k_n}(1-e^{-1})^{k_n+1}$ since
$k_n-\Floor{k_n}$ may oscillate between $0$ and $1$.
However, we do have $\Prob*{}{\tilde{T}_n > k_n} = \Theta({(1-e^{-1})^{k_n}}/{k_n})$,
where the Bachmann--Landau notation $u_n=\Theta(v_n)$ indicates that
there exist two positive constants
$c$ and $C$ such that $c\,v_n\leq u_n \leq Cv_n$ for $n$ large enough.
This approximation is sufficient for our purpose.

We may now prove Theorem \ref{thm:largestTree} using the first and second moment method that we already used for the largest degree.

\begin{proof}[Proof of Theorem \ref{thm:largestTree}]
  We apply Lemma \ref{lemmaMomentsMethod} to the exchangeable variables
  \[
    (X_n^{(1)},\ldots, X_n^{(n)})=(\tilde{T}_n^{(1)}-\alpha\log n, \ldots, \tilde{T}_n^{(n)}-\alpha\log n)\,, 
  \]
  with $m_n=\alpha \log\log n$ and $\beta=-1$.
  The first two points of the lemma are readily checked, since \eqref{eq:tail-tree} tells us that for
  \[
     \alpha = (1-\log(e-1))^{-1} = -(\log(1-e^{-1}))^{-1} 
  \]
  and any $\gamma>0$, we have for $k_n\defas\alpha(\log n-\gamma\log\log n)$
  \begin{align} \label{eqProbTtildegeqk}
    \Prob{\tilde{T}_n-\alpha\log n\geq -\gamma\alpha \log\log n} =
    \Prob{\tilde{T}_n\geq k_n} =
    \Theta\Big( \frac{(\log n)^{\gamma-1}}{n}\Big).
  \end{align}
  Thus, for all $\epsilon > 0$,
  \begin{mathlist}
    \item $n \Prob{\tilde{T}_n-\alpha\log n\geq (-1+\epsilon)\alpha \log\log n} \tendsto{} 0.$
    \item $n \Prob{\tilde{T}_n-\alpha\log n\geq (-1-\epsilon)\alpha \log\log n} \tendsto{} +\infty.$
  \end{mathlist}

  All that remains to check is the third point of the lemma.
  From now on, we fix $k_n=\alpha(\log n-(1+\epsilon)\log\log n)$ for some
  $\epsilon>0$, and for the sake of readability, we set
  $R_n \defas \Prob*{}{\tilde{T}_n\geq k_n}$.
  With this notation, given Remark \ref{rem:three-prime} we need to show
  \begin{equation} \label{eq:asymptotic-independence-largest-tree}
    \Prob*\big{\tilde{T}_n^{(1)}\geq k_n, \tilde{T}_n^{(2)}\geq k_n} \;\leq\; R_n^2+o(R_n^2).
  \end{equation}
  Since this is rather technical, we defer the complete proof to
  Lemma~\ref{lem:asymptotic-indep-largest-tree} in Appendix~\ref{app:technical}, and
  only outline the main ideas of the proof here.
  As in the study of the largest degree, we prove this by showing that the law
  of $\tilde{T}_n^{(2)}$ conditional on $\Set*{\tilde{T}_n^{(1)} \geq k_n}$ is
  close to its unconditional law.  We first prove that
  \[
    \Prob*\big{A_n,\,\tilde{T}_n^{(1)}\geq k_n, \tilde{T}_n^{(2)}\geq k_n} \;=\; o(R_n^2),
  \]
  where $A_n\defas\{\tilde{\Ti}_n^{(2)}\subset\tilde{\Ti}_n^{(1)}\}\sqcup\{\tilde{\Ti}_n^{(1)}\subset\tilde{\Ti}_n^{(2)}\}$
  is the event that one of the two vertices~$1$ and~$2$ is an ancestor of the other in the UA construction.
  We then show
  \[
    \Prob*\big{A_n^{\mathrm{c}},\,\tilde{T}_n^{(1)}\geq k_n, \tilde{T}_n^{(2)}\geq k_n}
    \;\leq\; R_n^2+o(R_n^2)\,, 
  \]
  where $A_n^{\mathrm{c}}$ denotes the complement of $A_n$.
  This is done by showing that, conditional on $\Set*{\tilde{T}_n^{(1)}=i}$, on
  the event $A_n^{\mathrm{c}}$ the process counting the number of vertices of
  the tree $\Ti_n^{(2)}$ in the UA construction behaves as a modified
  $\Upsilon_n$ process, which we essentially bound from above by $\Upsilon_{n-i}$.
  Therefore, $\tilde{T}_n^{(2)}$ can be compared with an independent variable with distribution $\tilde{T}_{n-i}$.
  Finally, we show that
  \[
    \sum_{i\geq k_n}\Prob{\tilde{T}_n^{(1)}=i}\Prob{\tilde{T}_{n-i}\geq k_n} \;\leq\;
    R_n^2+o(R_n^2),
  \]
  thereby proving \eqref{eq:asymptotic-independence-largest-tree} and concluding the proof of Theorem \ref{thm:largestTree}.
\end{proof}

\change{
\section{Concluding comments}
}

\subsection{Aldous's construction}
The UA construction described in Section \ref{secUAC} is reminiscent of
Aldous's construction \cite[Algorithm 2]{Ald90} of a uniform rooted
labeled tree as a function of $n$ uniform random variables on
$\{1,\dots,n\}$ -- an algorithm arguably simpler than the variant of
Joyal's bijection used in the proof of Theorem~\ref{thmConstructionUniformTrees}.
However, we could not find a simple way to couple both procedures so that
the forest obtained from the UA construction coincides with the one
obtained by removing decreasing edges in Aldous's uniform tree.

\subsection{A local limit}

The construction of $\MF{n}$ from a uniform random tree in 
Section~\ref{secUnifLabeledTrees} gives us a way to build an infinite forest as a
limiting object for the Moran forest, in a weak sense. Let us start by
describing the local weak limit of the uniform (rooted) random tree. Recall
that the local weak limit of a sequence of random graph $\mathcal{G}_n$ is a
(possibly infinite) random pointed graph $(\mathcal{G},u_{\mathcal{G}})$
such that for each finite radius $r\geq 1$, the $r$-neighborhood of a
uniformly chosen vertex $u_n\in \mathcal{G}_n$ converges in distribution
to the ball of radius $r$ around $u_{\mathcal{G}}$ in $\mathcal{G}$.
Consider the following infinite random tree:
\begin{itemize}
  \item Start from an infinite \emph{spine} of vertices
      $u_0,u_1,u_2,\dots$, with edges $(u_i\leftarrow u_{i+1})$ between
      subsequent vertices, directed toward the focal vertex $u_0$.
  \item Let independent Galton-Watson trees with Poisson($1$) offspring
      distribution start from each $u_i$, for $i\geq 0$, with edges
      directed from mothers to daughters.
\end{itemize}
The graph $\mathcal{T}_\infty$ described above is the local weak limit of the
random rooted uniform tree on $n$ labeled vertices \cite{Gri80}.
The root is informally placed at the end of the infinite spine.

In order to translate this result to the Moran forest, we need to remove
the decreasing edges of $\mathcal{T}_{\infty}$. To do so, we equip each
vertex $v$ with an independent uniform variable $V_v$ on $[0, 1]$, that
corresponds to the limiting renormalized label of $v$. The graph obtained 
by removing from $\mathcal{T}_\infty$ each edge $\vec{uv}$ such that 
$V_u > V_v$ can be understood as the local weak limit of the Moran
forest.

This construction can be used to derive some limiting results about the
local structure of the Moran forest. We can for instance recover
the limiting degree of a uniformly chosen vertex: it is clear that the
focal vertex $u_0$ in the construction above has degree 
\[
    D \;\overset{d}{=}\; \mathrm{Ber}(V_{u_0}) + \mathrm{Poisson}(1-V_{u_0}),
\]
as described in Proposition~\ref{propLimitLawDegree}. However, global
results such as the size of the largest tree or the largest degree
cannot be easily derived from this local weak limit.

\begin{remark}~
    \begin{mathlist}
    \item By definition, the local weak limit should be a.s.\ connected. Only
    the random tree that contains $u_0$ corresponds to the actual local
    weak limit of the Moran forest. It seems difficult to give meaning 
    to the other trees that are obtained in this procedure. They are
    connected to the tree containing the focal vertex through the labels
    $(V_v)$. This indicates that the tree that is adjacent to the focal
    tree is not simply an independent ``local limit tree'' for another 
    uniformly chosen vertex.
    
    \item Note that the local weak limit can be obtained in a
    simpler, more direct way. Indeed, Theorem~\ref{thmAsymptoticTreeSize}
    gives us the limiting size of the tree containing a uniformly chosen
    vertex, and we know from the UA construction that, conditional on
    the number of trees and their sizes, trees in the Moran forest are
    uniform attachment trees.  \qedhere
    \end{mathlist}
\end{remark}

\subsection{Possible extension}

An important property of the Moran model is that it is an exchangeable
population model: its distribution is invariant under re-labeling of the
vertices. General exchangeable population models are known as
\emph{Cannings models}, and just as for the Moran model it is possible
to associate a forest-valued process to any Cannings model.

A Cannings model is defined from an exchangeable vector $(\xi(1),
\dots, \xi(n))$ of non-negative integers verifying 
$\xi(1) + \dots + \xi(n) = n$. This vector encodes the offspring 
distribution of a population labeled by $\Set{1, \dots, n}$.
If $\xi(v) = 0$, we say that individual $v$ is dead.
Otherwise, it has $\xi(v) - 1$ children. It is clear that the number of
dead individuals is equal to the total number of children in the
population. We can thus assign a mother to each dead individual, in such a
way that the number of children of each live individual is $\xi(v) - 1$. 

Starting from any directed graph, we can now define a
transition as follows:
\begin{enumerate}
    \item Draw a vector distributed as $(\xi(1), \dots, \xi(n))$.
    \item Disconnect each dead vertex from all its neighbors.
    \item Assign a mother to each dead vertex, uniformly among all
        possibilities. For each dead vertex $v$, draw an edge from its
        mother to $v$.
\end{enumerate}
This defines a Markov chain on the set of rooted forests.
It is not hard to see that, if $(\xi(1), \dots, \xi(n))$ is the uniform
permutation of the vector $(2, 0, 1, \dots, 1)$, we recover the Markov
chain leading to the Moran forest described in Section~\ref{secModel}.

\enlargethispage{2ex}

Studying these general exchangeable forest processes is not the aim of
the current work. In particular, the techniques we used here rely
heavily on the construction of the Moran forest described in
Section~\ref{secUAC}, which cannot be easily adapted to more general
exchangeable forest processes.

\section*{Acknowledgments}

We thank Amaury Lambert for initiating the discussion that led us to consider
this model and for comments on the first version of this manuscript,
\change{and Justin Salez for interesting discussions on the link between
the uniform random rooted tree, its local limit and the Moran forest.
We are also grateful to the editor for pointing out the similarity
between the UA construction and Aldous's algorithm, and to two
anonymous referees for helpful comments, including the possibility of
adapting the definition of the Moran forest to Cannings
models.

This work was for the most part done while JJD worked in CIRB and LPSM.}

\phantomsection
\addcontentsline{toc}{section}{References}
\bibliographystyle{abbrvnat}
\bibliography{biblio} 

\pagebreak

\appendix

\section{Appendix}

\subsection{Proof of point (ii) of Proposition~\ref{propTailDegree}}
\label{appEndProofTailDegree}

We want to prove that, for all $K_n = o(\sqrt{n})$, there exists
$\epsilon_n = o(1)$ such that, for all $k \leq K_n$,
\[
  \Abs{\Prob{D_n \geq k} - \Prob{D \geq k}}
  \;\leq\; \epsilon_n\, \Prob{D \geq k} \,.
\]
Doing this directly from the expressions of $D_n$ and $D$
involves unappealing calculations.
To somewhat circumvent this, we make use of the simple expressions of the
probability generating functions $G_{D_n}$ and $G_{D}$. For this, let
\[
  \Delta_n(z) \;\defas\;
  \sum_{i \geq 0} \big(\Prob{D \geq i} - \Prob{D_n \geq i}\big) z^i, 
\]
so that the $k$-th derivative of $\Delta_n$ evaluated at $z = 0$ is
\[
  \Delta_n^{(k)}(0) \;=\; k!\, \big(\Prob{D \geq k} - \Prob{D_n \geq k}\big).
\]
Since $\Prob{D \geq k} \geq \frac{2/e}{(k + 1)!}$, we have to show that
for any given sequence $K_n = o(\sqrt{n})$,\linebreak
\[
  \Delta_n^{(k)}(0) \;=\; \frac{\epsilon_n}{k + 1}
\]
for some $\epsilon_n \to 0$ and all $k \leq K_n$.
Now, since for any non-negative integer-valued
random variable $X$,
\[
  \sum_{i \geq 0} \Prob{X \geq i}\,z^i \;=\; \frac{z\,\Expec{z^X} - 1}{z - 1}\,, 
\]
we can express $\Delta_n$ in terms of the generating functions of $D$ and
$D_n$, that is,
\[
  \Delta_n(z) \;=\; \mleft(1 + \frac{1}{z - 1}\mright)
  \big(G_D(z) - G_{D_n}(z)\big)\,.
\]
Moreover, we know from Proposition~\ref{propLimitLawDegree} that
\[
  G_D(z) =\; 2\,\frac{e^{z - 1} - 1}{z - 1} - 1 \;=\;
           2\sum_{i \geq 0} \frac{(z - 1)^i}{(i + 1)!}  - 1
\]
and from Proposition~\ref{propLawDegree} that
\begin{align*}
  G_{D_n}(z)
  &=\; 2 \mleft(1 - \frac{1}{n}\mright)
    \frac{\mleft(1 + \frac{z - 1}{n - 1}\mright)^n - 1}{z - 1} \;-\; 1 \\
  &=\; 2 \mleft(1 - \frac{1}{n}\mright)
    \sum_{i = 0}^{n - 1} \binom{n}{i + 1} \mleft(\frac{1}{n - 1}\mright)^{i + 1}
    (z - 1)^i \;-\; 1 \\
  &=\; 2 \sum_{i = 0}^{n - 1} \mleft(\prod_{\ell = 1}^i\frac{n - \ell}{n - 1}\mright)
  \frac{(z - 1)^i}{(i + 1)!} \;-\; 1\,, 
\end{align*}
where the empty product is~1. Therefore,
\[
  G_D(z) - G_{D_n}(z) =   
  \sum_{i \geq 0} A(n, i) \frac{(z - 1)^i}{(i + 1)!} \,, 
\]
where
\[
  A(n, i) = 2 \mleft[1 - \mleft(\prod_{\ell = 1}^i\frac{n - \ell}{n - 1}\mright)
  \Indic{i \leq n - 1}\mright]\,.
\]
Using that $A(n, 0) = A(n, 1) = 0$ and rearranging a bit, we obtain the
following expansion of $\Delta_n$ at $z = 1$:
\[
  \Delta_n(z) \;=\;
  \sum_{i \geq 1} \mleft(A(n, i) + \frac{A(n, i + 1)}{i + 2}\mright) 
                  \frac{(z - 1)^i}{(i + 1)!}\,, 
\]
from which we get
\[
  \Delta_n^{\!(k)}(0) \;=\;
  \sum_{i \geq k} \mleft(A(n, i) + \frac{A(n, i + 1)}{i + 2}\mright) 
                  \frac{(-1)^{i - k}}{(i - k)!\,(i + 1)} \,.
\]
Now, pick any $J_n = o(\sqrt{n})$ such that $K_n = o(J_n)$. For all $i < J_n$,
\[
  \Abs{A(n, i) + \frac{A(n, i + 1)}{i + 2}} \;\leq\;
  4 \mleft(1 - \prod_{\ell = 1}^{J_n}\frac{n - \ell}{n - 1}\mright) \;=\;
  \epsilon_n\,,
\]
with $\epsilon_n \to 0$, since
\[
  \prod_{\ell = 1}^{J_n}\frac{n - \ell}{n - 1} \;\geq\;
  \mleft(\frac{n - J_n}{n - 1}\mright)^{J_n} =\;
  \exp\mleft(-\tfrac{J_n^2}{n} + o\mleft(\tfrac{J_n^2}{n}\mright)\mright)\,.
\]

For $i \geq J_n$, we have
\[
  \Abs{A(n, i) + \frac{A(n, i + 1)}{i + 2}} \;\leq\; 4 \,.
\]
Combining these two upper bounds, we get
\begin{align*}
  \Abs{\Delta_n^{\!(k)}(0)} &\;\leq\;
  \sum_{i = k}^{J_n - 1} \frac{\epsilon_n}{(i - k)!\,(i + 1)} \;+\;
  \sum_{i \geq J_n} \frac{4}{(i - k)!\,(i + 1)} \\
  &\;\leq\;
  \frac{\epsilon_n\,C_1}{(k + 1)} \;+\; \frac{C_2}{(J_n + 1)}\,.
\end{align*}
Finally, since $K_n = o(J_n)$, we have for all $k \leq K_n$,
\[
  \frac{1}{J_n + 1} \;\leq\;
  \frac{1}{k + 1} \cdot
  \frac{K_n + 1}{J_n + 1}\,, 
\]
with $(K_n + 1)/(J_n + 1) = o(1)$. This concludes the proof.

Note that although we have been quite crude in that we have used the
triangle inequality on an alternating series, a more careful analysis would 
show that the $o(\sqrt{n})$ requirement on $K_n$ is in fact optimal.

\subsection{Technical lemmas used in the proof of Theorem~\ref{thm:largestTree}} \label{app:technical}

\begin{lemma} \label{lem:integral-equivalent}
 For any sequence $k_n\to\infty$ and any sequence of measurable maps $f_n:[0,1]\to\R$ such that for all $x\in [0,1]$, $(1-e^{-x} +f_n(x))\geq 0$ and $\sup_x\Abs{f_n(x)}=o(1/k_n)$, we have
  \[
  \int_{0}^{1} \big(1-e^{-x} + f_n(x)\big)^{k_n} \,dx \,\sim\, \frac{e}{k_n}(1-e^{-1})^{k_n+1}.
  \]
\end{lemma}
\begin{proof}
  Let us compute
  \begin{align*}
  \int_{0}^{1} \frac{(1-e^{-x} + f_n(x))^{k_n}}{(1-e^{-1})^{k_n}} k_n\,dx &= \int_{0}^{1} \left (1 - \frac{e^{1-x}-1}{e-1}+\frac{e}{e-1}f_n(x)\right )^{k_n} k_n\, dx\\
  &= \int_{0}^{k_n} \left (1 - \frac{y}{k_n}+g_n(y)\right )^{k_n} \frac{e-1}{1+(e-1)\frac{y}{k_n}}\, dy,
  \end{align*}
  where we used the change of variable $y=k_n(e^{1-x}-1)(e-1)^{-1}$, and defined the map $g_n$ as
  \[
    g_n(y) = \frac{e}{e-1}f_n\left (1-\log\big(1+\frac{y}{k_n}(e-1)\big)\right ).
  \]
  Since $(1 - \frac{y}{k_n}+g_n(y))^{k_n}\leq \exp(-y+\frac{e}{e-1}k_n\sup_x f_n(x))$, it follows from dominated convergence that 
  \[
  \int_{0}^{1} \frac{(1-e^{-x} + f_n(x))^{k_n}}{(1-e^{-1})^{k_n}} k_n\,dx \tendsto{n\to\infty} \int_0^{\infty}e^{-y}(e-1)\, dy = e-1,
  \]
  concluding the proof.
\end{proof}

\begin{lemma} \label{lem:asymptotic-indep-largest-tree}
  Let $\tilde{T}^{(v)}_n$ denote the size of the subtree descending from $v$ in
  the UA construction of $\MF{n}$.
  Then, for $\alpha =-1/\log(1-e^{-1})$ and any $\epsilon > 0$,
  letting $k_n = \alpha(\log n -(1+\epsilon)\log\log n)$ and
  $R_n = \Prob{\tilde{T}_n\geq k_n}$,
  \[
    \Prob*\big{\tilde{T}_n^{(1)}\geq k_n, \tilde{T}_n^{(2)}\geq k_n} \;\leq\; R_n^2+o(R_n^2).
  \]
\end{lemma}
\begin{proof}
  Let us denote by
  $A_n\defas\{\tilde{\Ti}_n^{(2)}\subset\tilde{\Ti}_n^{(1)}\}\sqcup\{\tilde{\Ti}_n^{(1)}\subset\tilde{\Ti}_n^{(2)}\}$
  the event that one of the vertices $1$ and $2$ is an ancestor of the other.
  We start by showing that
  \begin{equation} \label{eq:proba-conflit}
  \Prob*\big{A_n,\,\tilde{T}_n^{(1)}\geq k_n, \tilde{T}_n^{(2)}\geq k_n} = o(R_n^2).
  \end{equation}
  By exchangeability, we have
  \begin{align*}
  &\Prob*\big{A_n,\,\tilde{T}_n^{(1)}\geq k_n, \tilde{T}_n^{(2)}\geq k_n}\\
  &\qquad= 2\,\Prob*\big{\tilde{\Ti}_n^{(2)}\subset\tilde{\Ti}_n^{(1)},\,\tilde{T}_n^{(1)}\geq k_n, \tilde{T}_n^{(2)}\geq k_n}\\
  &\qquad=\sum_{i\geq k_n}\Prob*\big{\tilde{\Ti}_n^{(2)}\subset\tilde{\Ti}_n^{(1)},\, \tilde{T}_n^{(2)}\geq k_n\given\tilde{T}_n^{(1)}=i}\,\Prob{\tilde{T}_n=i}.
  \end{align*}
  Let us call the \emph{height} of a vertex the number of steps after it was
  added in the UA construction.
  Conditional on $\Set*{\tilde{T}_n^{(1)} = i}$ and on the heights of the vertices of
  $\tilde{\Ti}_n^{(1)}$ being $\ell_1>\ldots>\ell_i$,
  the height $L_2$ of vertex~$2$ is uniformly distributed
  on $\{0,\ldots n-1\}\setminus\{\ell_1\}$. Moreover, in order to have
  \[
    \{\tilde{\Ti}_n^{(2)}\subset\tilde{\Ti}_n^{(1)},\,\tilde{T}_n^{(2)} \geq k_n\}\,, 
  \]
  the height of vertex~$2$ must belong to $\{\ell_2, \ldots, \ell_{i-(k_n-1)}\}$,
  which happens with probability $\frac{i-k_n}{n-1}$.
  Therefore,
  \begin{align*}
  &\Prob*\big{A_n,\,\tilde{T}_n^{(1)}\geq k_n, \tilde{T}_n^{(2)}\geq k_n}\\
  &\qquad\leq\sum_{i\geq k_n}\Prob{\tilde{T}_n=i}\frac{i-k_n}{n-1}\\
  &\qquad=\frac{1}{n-1}\sum_{i > k_n}\Prob{\tilde{T}_n\geq i}.
  \end{align*}
  To show that this is small enough, we let $K_n \defas k_n+\alpha(\log n)^{\delta}$ with $0<\delta < \min(1,\epsilon)$ and $K_n'\defas 2\alpha\log n$, and crudely bound
  \begin{align*}
  \sum_{i > k_n}\Prob{\tilde{T}_n\geq i}&\leq (K_n-k_n)\Prob{\tilde{T}_n\geq k_n} + K'_n\,\Prob{\tilde{T}_n\geq K_n} + n\,\Prob{\tilde{T}_n\geq K'_n}.
  \end{align*}
  Now let us show that these three terms are negligible compared to $n R_n^2$.
  Recalling from~\eqref{eqProbTtildegeqk} that
  $R_n =\Theta\big(\frac{(\log n)^{\epsilon}}{n}\big)$, we have
  $n R_n^2=\Theta((\log n)^{2\epsilon}/n)$ and
  therefore
  \begin{itemize}
    \item $\displaystyle (K_n-k_n)\Prob{\tilde{T}_n\geq k_n} \sim \alpha(\log n)^{\delta}R_n =\Theta\Big(\frac{(\log n)^{\delta+\epsilon}}{n}\Big) = o(nR_n^2)$.
    \item $\displaystyle K'_n\,\Prob{\tilde{T}_n\geq K_n} =\Theta\big( \log n\, R_ne^{-(\log n)^{\delta}}\big) = o(R_n) = o(nR_n^2)$.
    \item $\displaystyle  n\,\Prob{\tilde{T}_n\geq K'_n} =\Theta \Big(n \frac{n^{-2}}{\log n}\Big) = o(1/n) = o(nR_n^2)$.
  \end{itemize}
  As a result, \eqref{eq:proba-conflit} is proven and it remains to show that
  \[
    \Prob*\big{A_n^{\mathrm{c}},\,\tilde{T}_n^{(1)}\geq k_n, \tilde{T}_n^{(2)}\geq k_n} \leq R_n^{2}+o(R_n^{2}),
  \]
  where $A_n^{\mathrm{c}}$ denotes the complement of $A_n$.
  We now fix $n\geq 1$, $i\geq k_n$, and a finite sequence $n-1\geq \ell_1>\ldots>\ell_i\geq 0$.
  Let us write $B$ for the event that $\tilde{\Ti}_n^{(1)}$ contains exactly the vertices with heights $\ell_1>\ldots>\ell_i$.
  Conditional on $B$, let us examine the distribution of $\tilde{\Ti}_n^{(2)}$.
  Recall that the height $L_2$ of vertex $2$ is uniformly distributed on $\{0,\ldots n-1\}\setminus\{\ell_1\}$.
  In the UA construction, define $\mathbb{T}$ as the tree obtained by starting
  from a root arrived at height $L_2$ and allowing the attachment of
  a vertex with height~$\ell$ to $\mathbb{T}$ only if $\ell\notin\{\ell_1,\ldots \ell_i\}$.
  Then, on the event $A_n^{\mathrm{c}}$, this tree must coincide with $\Ti_n^{(2)}$, and so
  \[
  \Prob*\big{A_n^{\mathrm{c}},\,\tilde{T}_n^{(2)}\geq k_n\given B} \;=\;
  \Prob*\big{A_n^{\mathrm{c}},\,\Abs{\mathbb{T}}\geq k_n\given B}.
  \]
  From the UA construction, for any $\ell\notin\Set{\ell_1,\ldots,\ell_i}$, conditional on $B\cap\Set*{L_2=\ell}$, we can
  describe~$\Abs{\mathbb{T}}$ using the process
  $(\widetilde{\Upsilon}_\ell(m),\,0\leq m\leq \ell)$ defined by
  \begin{itemize}
    \item $\widetilde{\Upsilon}_\ell(0) = 1$.
    \item For all $0< m \leq \ell$, 
      $\widetilde{\Upsilon}_\ell(m) - \widetilde{\Upsilon}_\ell(m - 1) \in \Set{0, 1}$
      and, conditional on \newline $\Set*{\widetilde{\Upsilon}_\ell(m-1)=j}$,
      $\widetilde{\Upsilon}_\ell(m) = j+1$ with probability
    \[
    \begin{cases}
    \dfrac{j}{n-1-J_m} &\text{ if } \ell-m \notin \{\ell_1,\ldots,\ell_i\}\\
    0 &\text{ if } \ell-m\in \{\ell_1,\ldots,\ell_i\},
    \end{cases}
    \]
    where
    $J_m=\Abs{\{\ell_1,\ldots,\ell_i\}\cap\{\ell-m,\ldots,n\}}$
    is the number of vertices of $\tilde{T}_n^{(1)}$ with height greater
    than $\ell-m$ in the UA construction.
  \end{itemize}
  With this definition, for any $\ell\notin\Set{\ell_1,\ldots,\ell_i}$,
  conditional on $B\cap\{L_2=\ell\}$, we have by construction
  \smash{$\Abs{\mathbb{T}}\overset{d}{=}\tilde{\Upsilon}_\ell(\ell)$}.
  Now, note that the probability of increasing is always bounded by ${j/(n-1-i)}$.
  Therefore, $\widetilde{\Upsilon}_\ell$ can be coupled with $\Upsilon_{n-i}$
  in such a way that, for all $0\leq m\leq \ell < n-i$,
  \[
    \widetilde{\Upsilon}_\ell(m)\leq\Upsilon_{n-i}(m).
  \]
  For $\ell\geq n-i$, we use the crude
  bound $\Prob*{}{\tilde{\Upsilon}_\ell(\ell)\geq k_n} \leq
  \Expec*{}{\tilde{\Upsilon}_\ell(\ell)}/k_n$.
  Using the same reasoning as in Lemma \ref{lemmaDiscreteYuleExpect}, we get
  \[
   \Expec*{}{\tilde{\Upsilon}_\ell(\ell)}\leq (1+\frac{1}{n-i-1})^{n-i-1} \leq e\,.
  \]
  We thus have
  \begin{align}
    \Prob*\big{A_n^{\mathrm{c}},\,\Abs{\mathbb{T}}\geq k_n\given B} \;&\leq\; \Prob*\big{L_2\notin\Set{\ell_1,\ldots,\ell_i},\,\Abs{\mathbb{T}}\geq k_n\given B} \notag\\
    &=\;\frac{1}{n-1} \;\sum_{\mathclap{\substack{\ell=0\\ \ell\notin\Set{\ell_1,\ldots\ell_i}}}}^{n-1}
    \;\Prob*\big{\tilde{\Upsilon}_\ell(\ell)\geq k_n}, \notag\\
    &\leq\; \frac{ei}{k_n(n-1)} + \frac{1}{n-1}\sum_{\ell=0}^{n-i-1} \Prob{\Upsilon_{n-i}(\ell)\geq k_n} \label{eq:increasing-1}\\[1ex]
  &=\; \frac{ei}{k_n(n-1)} + \frac{n-i}{n-1}\Prob{\tilde{T}_{n-i}\geq k_n}. \label{eq:increasing-2}
  \end{align}
  Since this bound depends on the set $\{\ell_1,\ldots,\ell_i\}$ only via its cardinality $i$, one can integrate with respect to the distribution of $\Ti_n^{(1)}$ to get
  \[
  \Prob*\big{A^c_n,\,\Abs{\mathbb{T}}\geq k_n\given \tilde{T}_n^{(1)}=i} \;\leq\;
  \frac{ei}{k_n(n-1)} + \frac{n-i}{n-1}\Prob{\tilde{T}_{n-i}\geq k_n}.
  \]
  Finally, because $\Upsilon_{n-(i+1)}(\ell) \overset{d}{\geq} \Upsilon_{n-i}(\ell)$,
  the expression~\eqref{eq:increasing-1} -- and
  therefore~\eqref{eq:increasing-2} -- is nondecreasing in $i$, and we have
  \begin{align}
  &\Prob*\big{A_n^{\mathrm{c}},\,\tilde{T}_n^{(1)}\geq k_n, \tilde{T}_n^{(2)}\geq k_n} \notag\\
  &\qquad \leq \sum_{i\geq k_n}\Prob{\tilde{T}_n= i}\left(\frac{ei}{k_n(n-1)} + \frac{n-i}{n-1}\Prob{\tilde{T}_{n-i}\geq k_n}\right) \notag\\
  &\qquad  \leq \sum_{i=k_n}^{K_n}\Prob{\tilde{T}_n= i}\left(\frac{e\,K_n}{k_n(n-1)} + \frac{n-K_n}{n-1}\Prob{\tilde{T}_{n-K_n}\geq k_n}\right) \label{eq:bound-moche-1}\\
  &\qquad  \quad + \sum_{i\geq K_n}\Prob{\tilde{T}_n= i}\left(\frac{ei}{k_n(n-1)} + \frac{n-i}{n-1}\Prob{\tilde{T}_{n-i}\geq k_n}\right), \label{eq:bound-moche-2}
  \end{align}
  for any sequence $K_n\geq k_n$.  
  Letting $K_n \defas\alpha(\log n)^{1+\epsilon/2}$, we then show that
  \eqref{eq:bound-moche-1} is asymptotically no greater than $R_n^{2}$, and
  that \eqref{eq:bound-moche-2} is negligible compared to $R_n^{2}$.
  Indeed, \eqref{eq:bound-moche-1} is bounded from above by
  \[
  R_n\left(\frac{e\,K_n}{k_n(n-1)} + \frac{n-K_n}{n-1}\Prob{\tilde{T}_{n-K_n}\geq k_n}\right).
  \]
  Now note that $\frac{e\,K_n}{k_n(n-1)}=O(\frac{(\log n)^{\epsilon/2}}{n})=o(R_n)$,
  and that since $n-K_n \sim n$, we have
  $k_n=o(\sqrt{n-K_n})$. Therefore, using \eqref{eq:tail-tree} we get
  $\Prob{\tilde{T}_{n-K_n}\geq k_n} \sim R_n$.
  Finally, up to a multiplicative constant, \eqref{eq:bound-moche-2} is bounded from above by
  \[
  \Prob{\tilde{T}_{n}\geq K_n} =\Theta\Big(\frac{n^{-(\log n)^{\epsilon/2}}}{K_n}\Big) = o(n^{-2})=o(R_n^{2}).
  \]
  Putting everything together, we have proved that
  \[
  \Prob*\big{\tilde{T}_n^{(1)}\geq k_n, \tilde{T}_n^{(2)}\geq k_n} \;\leq\;
  R_n^{2}+o(R_n^{2}),
  \]
  which concludes the proof.
\end{proof}

\end{document}